\documentclass[12pt,twoside]{amsart}

\usepackage{amsmath}
\usepackage{amssymb}
\usepackage{amscd}
\usepackage{graphics}
\usepackage{graphicx}

\renewcommand{\bar}{\overline}

\newcommand{\AAA}{\mathbb{A}}

\newcommand{\CC}{\mathbb{C}}

\newcommand{\PP}{\mathbb{P}}
\newcommand{\QQ}{\mathbb{Q}}
\newcommand{\RR}{\mathbb{R}}

\newcommand{\ZZ}{\mathbb{Z}}

\newcommand{\Qbar}{\bar{\QQ}}

\newcommand{\Cv}{\CC_v}

\newcommand{\calM}{{\mathcal M}}

\newcommand{\frakK}{{\mathfrak{K}}}

\newcommand{\frakh}{{\mathfrak{h}}}

\newcommand{\Dbar}{\bar{D}}

\newcommand{\dsps}{\displaystyle}

\theoremstyle{plain}
\newtheorem{thm}{Theorem}[section]
\newtheorem{prop}[thm]{Proposition}
\newtheorem{cor}[thm]{Corollary}
\newtheorem{lemma}[thm]{Lemma}
\newtheorem{conj}{Conjecture}

\theoremstyle{definition}
\newtheorem{defin}[thm]{Definition}
\newtheorem{alg}[thm]{Algorithm}

\theoremstyle{remark}
\newtheorem{remark}[thm]{Remark}

\numberwithin{equation}{section}

\title[Points of small height]
{Computing Points of Small Height for Cubic Polynomials}
\date{September 25, 2008; revised November 25, 2008}


\author[Benedetto]{Robert~L. Benedetto}
\address[Benedetto, Dickman, Joseph, Krause, Zhou]
	{Amherst College \\ Amherst, MA 01002}
\address[Rubin]{Johns Hopkins University \\ Baltimore, MD 21218}
\email[Benedetto]{rlb@cs.amherst.edu}
\author[Dickman]{Benjamin Dickman}
\email[Dickman]{bdickman08@amherst.edu}
\author[Joseph]{Sasha Joseph}
\email[Joseph]{sjoseph08@amherst.edu}
\author[Krause]{Benjamin Krause}
\email[Krause]{bkrause10@amherst.edu}
\author[Rubin]{Daniel Rubin}
\email[Rubin]{drubin9@jhu.edu}
\author[Zhou]{Xinwen Zhou}
\email[Zhou]{xzhou09@amherst.edu}

\begin{document}

\newcounter{bean}
\newcounter{sheep}

\begin{abstract}
Let $\phi\in \QQ[z]$ be a polynomial of
degree $d$ at least two.  The associated canonical height $\hat{h}_{\phi}$
is a certain real-valued function on $\QQ$ that returns
zero precisely at preperiodic rational points of $\phi$.  Morton
and Silverman conjectured in 1994 that the number of such points
is bounded above by a constant depending only
on $d$.  A related conjecture claims that at non-preperiodic rational
points, $\hat{h}_{\phi}$ is bounded below 
by a positive constant (depending only on $d$) times
some kind of height of $\phi$ itself.  In this paper, we provide support
for these conjectures in the case $d=3$ by computing the set
of small height points for several billion cubic polynomials.
\end{abstract}

\maketitle

Let $\phi(z)\in\QQ[z]$ be a polynomial
with rational coefficients.  Define $\phi^0(z)=z$, and for
every $n\geq 1$, let $\phi^n(z)=\phi\circ\phi^{n-1}(z)$; that is,
$\phi^n$ is the $n$-th iterate of $\phi$ under composition.
A point $x$ is said to be {\em periodic}
under $\phi$ if there is an integer $n\geq 1$ such that
$\phi^n(x)=x$.  In that case, we say $x$ is $n$-periodic;
the smallest such positive integer $n$ is called the
{\em period} of $x$.  More generally, $x$ is
{\em preperiodic} under
$\phi$ if there are integers $n>m\geq 0$ such that
$\phi^n(x)=\phi^m(x)$; equivalently, $\phi^m(x)$ is
periodic for some $m\geq 0$.

In 1950, using the theory of arithmetic heights,
Northcott proved that if $\deg\phi\geq 2$, then
$\phi$ has only finitely many preperiodic points in $\QQ$.
(In fact, his result applied far more generally,
to morphisms of $N$-dimensional
projective space over any number field.)
In 1994, motivated by
Northcott's result and by analogies to torsion
points of elliptic curves
(for which uniform bounds were proven by Mazur \cite{Maz}
over $\QQ$ and by Merel \cite{Mer} over arbitrary number fields),
Morton and Silverman proposed
a dynamical Uniform Boundedness Conjecture
\cite{MS1,MS2}.  Their conjecture
applied to the same general setting as Northcott's Theorem,
but we state it here only for polynomials over $\QQ$.

\begin{conj}[Morton, Silverman 1994]
\label{conj:ubc}
For any $d\geq 2$, there is a constant $M=M(d)$ such that
no polynomial $\phi\in\QQ[z]$ of degree $d$ has more
than $M$ rational preperiodic points.
\end{conj}

Thus far only partial results towards
Conjecture~\ref{conj:ubc} have been proven.  Several authors
\cite{MS1,MS2,Nar,Pez1,Zieve} have bounded the period
of a rational periodic point in terms of the smallest
prime of good reduction (see Definition~\ref{def:goodred}).  Others
\cite{FPS,Man,Mor1,Mor2,Poo} have proven that polynomials of degree two
cannot have rational periodic points of certain periods by
studying the set of rational points on an associated dynamical
modular curve; see also \cite[Section~4.2]{Sil}.
A different method, introduced in \cite{CaGol}
and generalized and sharpened in \cite{Ben9},
gave (still non-uniform) bounds for
the number of preperiodic points by
taking into account {\em all} primes, including those of bad reduction.


In a related vein,
the canonical height function $\hat{h}_{\phi}:\QQ\rightarrow [0,\infty)$
satisfies the functional
equation $\hat{h}_{\phi}(\phi(z)) = d\cdot \hat{h}_{\phi}(z)$,
where $d=\deg\phi$, and it has the property that
$\hat{h}_\phi(x)=0$ if and only if
$x$ is a preperiodic point of $\phi$;
see Section~\ref{sect:gchts}.
Meanwhile, if we consider $\phi$ itself as a point in
the appropriate moduli space of all polynomials of degree $d$,
we can also define $h(\phi)$ to be the arithmetic height
of that point;
see \cite[Section~4.11]{Sil}.  For example,
the height of the quadratic polynomial
$\phi(z) = z^2 + \frac{m}{n}$ is
$h(\phi):=h(\frac{m}{n}) = \log\max\{|m|,|n|\}$;
a corresponding height for cubic polynomials appears
in Definition~\ref{def:cubht}.   Again by analogy with
elliptic curves, we have the following conjecture,
stating that
the canonical height of a nonpreperiodic rational point
cannot be too small in comparison to $h(\phi)$;
see \cite[Conjecture~4.98]{Sil} for a more general
version.

\begin{conj}
\label{conj:ht}
Let $d\geq 2$.  Then there is a positive constant $M'=M'(d)>0$
such that for any polynomial $\phi\in\QQ[z]$ of degree $d$ and any
point $x\in\QQ$ that is not preperiodic for $\phi$, we have
$\hat{h}_{\phi}(x)\geq M' h(\phi)$.
\end{conj}

Just as Conjecture~\ref{conj:ubc}
says that any preperiodic rational point must land on a repeated value
after a bounded number of iterations, Conjecture~\ref{conj:ht}
essentially says that the size of a non-preperiodic rational
point must start to explode within a bounded number of iterations.
Some theoretical evidence for Conjecture~\ref{conj:ht}
appears in \cite{Bak,Ing},
and computational evidence when $d=2$
appears in \cite{Gil}.
The smallest known value of
$\hat{h}_{\phi}(x)/h(\phi)$
for $d=2$
occurs for $x=\frac{7}{12}$ under
$\phi(z) = z^2 - \frac{181}{144}$; the first few iterates are
$$\frac{7}{12} \mapsto
\frac{-11}{12} \mapsto
\frac{-5}{12} \mapsto
\frac{-13}{12} \mapsto
\frac{-1}{12} \mapsto
\frac{-5}{4} \mapsto
\frac{11}{36} \mapsto
\frac{-377}{324} \mapsto
\frac{2445}{26244} \mapsto \cdots.
$$
(This example was found in \cite{Gil} by a computer search.)
The small canonical height ratio
$\hat{h}_{\phi}(\frac{7}{12})/\log(181)\approx .0066$ 
makes precise the observation that although the numerators
and denominators of the iterates eventually explode in size,
it takes several iterations for the explosion to get underway.

In this paper, we
investigate {\em cubic} polynomials with rational coefficients.
We describe an algorithm
to find preperiodic and small height rational points of
such maps, and we present the resulting data, which supports
both conjectures.  In particular, after checking the fourteen
billion cubics with coefficients of smallest height, we found
none with more than eleven rational preperiodic points; those
with exactly ten or eleven are listed in Table~\ref{tab:11pts}.
Meanwhile, as regards Conjecture~\ref{conj:ht},
the smallest height ratio
$\frakh_{\phi}(x) := \hat{h}_{\phi}(x)/h(\phi)$
we found was about $.00025$, for
$\phi(z) = -\frac{25}{24}z^3 + \frac{97}{24}z + 1$ and the
point $x=-\frac{7}{5}$, with orbit
$$-\frac{7}{5} \mapsto
-\frac{9}{5} \mapsto
-\frac{1}{5} \mapsto
\frac{1}{5} \mapsto
\frac{9}{5} \mapsto
\frac{11}{5} \mapsto
-\frac{6}{5} \mapsto
-\frac{41}{20} \mapsto
\frac{4323}{2560} \mapsto
\ldots .$$
More importantly, although
we found quite 
a few cubics throughout the search
with a nonpreperiodic point of height ratio
less than $.001$,
only nine (listed in Table~\ref{tab:smallht}) gave $\frakh_{\phi}(x)<.0007$,
and the minimal one above was found early in the search.  Thus,
our data suggests that Conjecture~\ref{conj:ht} is true for cubic
polynomials, with $M'(3)=.00025$.

The outline of the paper is as follows.  In Section~\ref{sect:gchts}
we review heights, canonical heights,
and local canonical heights.
In Section~\ref{sect:lchts} we state and prove formulas
for estimating local canonical heights accurately
in the case of polynomials.
In Section~\ref{sect:fjsets}, we discuss
filled Julia sets (both complex and non-archimedean),
and in Section~\ref{sect:cubics} we consider cubics specifically.
Finally, we describe our search algorithm in Section~\ref{sect:alg}
and present the resulting data in Section~\ref{sect:data}.

Our exposition does not assume any background
in either dynamics or arithmetic heights, but
the interested reader is referred to
Silverman's text \cite{Sil}.
For
more details on non-archimedean filled Julia sets
and local canonical heights, see \cite{Ben9,CaGol,CaSil}.

\section{Canonical Heights}
\label{sect:gchts}

Denote by $M_{\QQ}$ the usual set
$\{|\cdot|_{\infty},|\cdot|_2,|\cdot|_3,|\cdot|_5,\ldots\}$
of absolute values (also called places) of $\QQ$,
normalized to satisfy the product formula
$$ \prod_{v\in M_\QQ}  |x|_v = 1
\qquad \text{for any nonzero } x\in\QQ^{\times} .$$
(See \cite[Chapters~2--3]{Gou}
or \cite[Chapter~1]{Kob}, for example,
for background on absolute values.)
The standard (global) {\em height function} on $\QQ$ is the function
$h:\QQ\to\RR$ given by
$h(x) := \log\max\{|m|_{\infty},|n|_{\infty}\}$, if we write
$x=m/n$ in lowest terms.  Equivalently,
\begin{equation}
\label{eq:hsumlog}
h(x) = \sum_{v\in M_\QQ} \log\max\{1,|x|_v\}
\qquad \text{for any }x\in\QQ .
\end{equation}
Of course, $h$ extends to the algebraic
closure $\Qbar$ of $\QQ$; see \cite[Section~3.1]{Lang},
\cite[Section~B.2]{HS}, or \cite[Section~3.1]{Sil}.
The height function satisfies two important properties.
First, 
for any polynomial $\phi(z)\in\QQ[z]$,
there is a constant $C=C(\phi)$ such that
\begin{equation}
\label{eq:hclose1}
\big|h(\phi(x)) - d\cdot h(x)\big| \leq C
\qquad \text{for all } x\in\Qbar,
\end{equation}
where $d=\deg\phi$.
Second, if we restrict $h$ to $\QQ$,
then for any bound $B\in\RR$,
\begin{equation}
\label{eq:htfin}
\{x\in\QQ : h(x) \leq B\} \quad \text{ is a  finite set}.
\end{equation}

For any fixed polynomial $\phi\in\QQ[z]$
(or more generally, rational function) of degree $d\geq 2$,
the {\em canonical height} function
$\hat{h}_{\phi}:\Qbar\to\RR$  for $\phi$ is given by
$$\hat{h}_{\phi}(x) := \lim_{n\to\infty} d^{-n} h(\phi^n(x)),$$
and it satisfies the functional equation
\begin{equation}
\label{eq:hfeqn}
\hat{h}_{\phi}(\phi(x)) = d\cdot \hat{h}_{\phi}(x)
\qquad \text{for all } x\in\Qbar.
\end{equation}
(The convergence of the limit and the functional equation
follow fairly easily from \eqref{eq:hclose1}.)
In addition,
there is a constant $C'=C'(\phi)$ such that
\begin{equation}
\label{eq:hclose2}
\big|\hat{h}_{\phi}(x) - h(x)\big| \leq C'
\qquad \text{for all } x\in\Qbar.
\end{equation}
Northcott's Theorem \cite{Nor} follows because
properties \eqref{eq:htfin}, \eqref{eq:hfeqn}, and \eqref{eq:hclose2}
imply
that for any $x\in\QQ$ (in fact, for any $x\in\Qbar$), 
$\hat{h}(x)=0$ if and only if $x$ is preperiodic under $\phi$.

For our computations,
we will need to compute $\hat{h}_{\phi}(x)$ rapidly
and accurately.
Unfortunately, the constants $C$ and $C'$ in
inequalities~\eqref{eq:hclose1} and~\eqref{eq:hclose2}
given by the general theory are rather weak and are rarely
described explicitly.  The goal of Section~\ref{sect:lchts}
will be to improve these constants, using
local canonical heights.

\begin{defin}
\label{def:lhdef}
Let $K$ be a field with absolute value $v$.  We denote by
$\Cv$ the completion of an algebraic closure of $K$.  The
function $\lambda_v:\Cv\rightarrow [0,\infty)$ given by
$$\lambda_v(x) := \log \max\{1,|x|_v\}$$
is called the
{\em standard local height} at $v$.
If $\phi(z)\in K[z]$ is a polynomial of degree $d\geq 2$,
the associated {\em local canonical height}
is the function
$\hat{\lambda}_{v,\phi}:\Cv\rightarrow [0,\infty)$
given by 
\begin{equation}
\label{eq:lhdef}
 \hat{\lambda}_{v,\phi}(x) :=
\lim_{n\to\infty} d^{-n} \lambda_v \big(\phi^n(x)\big).
\end{equation}
\end{defin}

According to \cite[Theorem~4.2]{CaGol}, the limit
in~\eqref{eq:lhdef} converges, so that the definition
makes sense.  It is immediate that
$\hat{\lambda}_{v,\phi}$ satisfies the functional
equation
$\hat{\lambda}_{v,\phi}(\phi(x)) = d\cdot \hat{\lambda}_{v,\phi}(x)$.
In addition, it is well known
that $\hat{\lambda}_{v,\phi}(x)-\lambda_v(x)$
is bounded independent of $x\in\Cv$; we shall prove a
particular bound in Proposition~\ref{prop:lhest} below.

Formula~\eqref{eq:lhdef}
of Definition~\ref{def:lhdef} is specific to
polynomials.  For a rational function
$\phi=f/g$, where $f,g\in K[z]$ are
coprime polynomials and $\max\{\deg f,\deg g\}=d\geq 2$,
the correct functional equation for
$\hat{\lambda}_{v,\phi}$ is
$\hat{\lambda}_{v,\phi}(\phi(x)) =
d\cdot \hat{\lambda}_{v,\phi}(x) - \log|g(x)|_v$.

Of course, 
formula~\eqref{eq:hsumlog} may now be writen as
$h(x) = \sum_{v\in M_{\QQ}} \lambda_v(x)$ for any $x\in\QQ$.
The local canonical heights provide a similar decomposition
for $\hat{h}_{\phi}$, as follows.

\begin{prop}
\label{prop:locglob}
Let $\phi(z)\in \QQ[z]$ be a polynomial of degree $d\geq 2$.
Then for all $x\in\QQ$,
$$\hat{h}_{\phi}(x) =
\sum_{v\in M_{\QQ}} \hat{\lambda}_{v,\phi}(x).$$
\end{prop}

\begin{proof}
See \cite[Theorem~2.3]{CaSil}, which
applies to arbitrary number fields, with appropriate modifications.
\end{proof}

Often, the local canonical height
$\hat{\lambda}_{v,\phi}$ exactly coincides with the
standard local height $\lambda_v$;
this happens precisely at the places of
{\em good reduction} for $\phi$.  Good reduction
of a map $\phi$
was first defined in \cite{MS1};
see also \cite[Definition~2.1]{Ben9}.
For polynomials, it is well known
(e.g., see \cite[Example~4.2]{MS2}) that
those definitions are equivalent to the following.

\begin{defin}
\label{def:goodred}
Let $K$ be a field with
absolute value $v$, and
let $\phi(z)=a_d z^d + \cdots + a_0 \in K[z]$
be a polynomial of degree $d\geq 2$.
We say that $\phi$ has {\em good reduction} at $v$
if
\begin{enumerate}
\item $v$ is non-archimedean,
\item $|a_i|_v \leq 1$ for all $i=0,\ldots, d$, and
\item $|a_d|_v = 1$.
\end{enumerate}
Otherwise, we say $\phi$ has {\em bad reduction} at $v$.
\end{defin}

Note
that if $K=\QQ$ (or more generally, if $K$ is a global field),
a polynomial $\phi\in K[z]$ has bad reduction at only finitely many
places $v\in M_K$.
As claimed above, we have the following result,
proven in, for example, \cite[Theorem~2.2]{CaGol}.

\begin{prop}
\label{prop:goodred}
Let $K$ be a field with absolute value $v$, and
let $\phi(z)\in K[z]$ be a polynomial of degree $d\geq 2$
with good reduction at $v$.  Then
$\hat{\lambda}_{v,\phi} = \lambda_v$.
\end{prop}

For more background on heights and canonical heights,
see \cite[Section~B.2]{HS}, \cite[Chapter~3]{Lang}, 
or \cite[Chapter~3]{Sil};
for local canonical heights, see
\cite{CaGol} or \cite[Section~2]{CaSil}.

\section{Computing Local Canonical Heights}
\label{sect:lchts}


\begin{prop}
\label{prop:lhest}
Let $K$ be a field with absolute value $v$,
let $\phi(z)\in K[z]$ be a polynomial of degree $d\geq 2$, and
let $\hat{\lambda}_{v,\phi}$ be the associated
local canonical height.
Write $\phi(z)=a_d z^d + \cdots + a_1 z + a_0
=a_d(z-\alpha_1)\cdots (z-\alpha_d)$, with
$a_i\in K$, $a_d\neq 0$, and $\alpha_i\in\Cv$
Let $A=\max\{|\alpha_i|_v : i=1,\ldots, d\}$
and $B=|a_d|_v^{-1/d}$, and
define real constants $c_v, C_v\geq 1$ by
$$c_v = \max\{1,A,B\}
\qquad \text{and} \qquad
C_v = \max\{1,|a_0|_v, |a_1|_v,\ldots, |a_d|_v\}$$
if $v$ is non-archimedean, or
$$
c_v = \max\{1,A+B\}
\qquad \text{and} \qquad
C_v = \max\{1,|a_0|_v + |a_1|_v + \ldots+ |a_d|_v\}$$
if $v$ is archimedean.
Then for all $x\in\Cv$,
$\dsps
\frac{-d\log c_v}{d-1} \leq
\hat{\lambda}_{v,\phi}(x) - \lambda_v(x) \leq
\frac{\log C_v}{d-1}$.
\end{prop}

\begin{proof}
First, we claim that
$\lambda_v(\phi(x)) - d\lambda_v(x) \leq \log C_v$
for any $x\in\Cv$.
To see this, if $|x|_v\leq 1$, then $|\phi(x)|_v\leq C_v$, and
the desired inequality follows.  If $|x|_v>1$ and $|\phi(x)|_v\leq 1$,
the inequality holds because $C_v\geq 1$.  Finally, if $|x|_v>1$
and $|\phi(x)|_v>1$, then the claim follows from the observation
that
$$\Big| \frac{\phi(x)}{x^d}\Big|_v = 
\big| a_d + a_{d-1} x^{-1} + \cdots + a_0 x^{-d} \big|_v
\leq C_v.$$

Next, we claim that 
$\lambda_v(\phi(x)) - d\lambda_v(x) \geq -d\log c_v$
for any $x\in\Cv$.
%
If $|x|_v\leq c_v$, then
$\lambda_v(x)\leq \log c_v$ because $c_v\geq 1$;
the desired inequality is therefore immediate
from the fact that 
$\lambda_v(\phi(x))\geq 0$.
If $|x|_v > c_v$, then
$$\lambda_v(\phi(x)) - d\lambda_v(x)
=\lambda_v(\phi(x)) - d\log |x|_v
\geq \log |\phi(x)|_v - d\log |x|_v,$$
by definition of $\lambda_v$ and because $|x|_v>c_v\geq 1$.
To prove the claim, then, it suffices to show
that $|\phi(x)|_v \geq (|x|_v/c_v)^d$ for $|x|_v> c_v$.

If $v$ is non-archimedean,
then $|x-\alpha_i|_v = |x|_v$ for all $i=1,\ldots, d$,
since $|x|_v>A\geq |\alpha_i|_v$.
Hence, $|\phi(x)|_v = |a_d|_v |x|_v^d = (|x|_v/B)^d
\geq (|x|_v/c_v)^d$.
If $v$ is archimedean, 
then
$$\frac{|x-\alpha_i|_v}{|x|_v}
\geq 1 - \frac{|\alpha_i|_v}{|x|_v}
\geq 1- \frac{A}{A+B} = \frac{B}{A+B}
\qquad\text{for all } i=1,\ldots, d.$$
Thus, $|\phi(x)|_v \geq |a_d|_v (B|x|_v/(A+B))^d
= (|x|_v/(A+B))^d\geq (|x|_v/c_v)^d$, as claimed.


To complete the proof, we compute
\begin{align*}
\hat{\lambda}_{v,\phi}(x) - \lambda_v(x) & =
\lim_{n\to\infty} \frac{1}{d^n} \lambda_v\big(\phi^n(x)\big) - \lambda_v(x)
=
\lim_{n\to\infty} \sum_{j=0}^{n-1}
\frac{1}{d^j} \Big[\frac{1}{d} \lambda_v\big(\phi^{j+1}(x)\big) 
- \lambda_v\big(\phi^{j}(x)\big) \Big]
\\
& \geq \lim_{n\to\infty} \sum_{j=0}^{n-1} -\frac{1}{d^j}\log c_v
= -\log c_v \sum_{j=0}^{\infty} \frac{1}{d^j}
= \frac{-d\log c_v}{d-1}.
\end{align*}
Similarly, 
$\hat{\lambda}_{v,\phi}(x) - \lambda_v(x) \leq (\log C_v)/(d-1)$.
\end{proof}

\begin{remark}
\label{rem:sij}
The proof above is just an explicit version
of \cite[Theorem~5.3]{CaSil},
giving good bounds for
$1$, $1/z^d$, $\phi(z)$, and $\phi(z)/z^d$
in certain cases--- e.g.
a lower bound for $|\phi(x)/x^d|_v$
when $|x|_v$ is large.  These are precisely the
four functions $\{s_{ij}\}_{ i,j\in\{0,1\} }$
in \cite{CaSil}.
\end{remark}

\begin{remark}
\label{rem:Adef}
If $v$ is non-archimedean, the quantity $A=\max\{|\alpha_i|_v\}$
can be computed directly from the coefficients of $\phi$.
Specifically,
$$A = \max\{|a_j/a_d|_v^{1/(d-j)}:0\leq j \leq d-1\}.$$
This identity is easy to verify
by recognizing $(-1)^{d-j} a_j/a_d$
as the $(d-j)$-th symmetric polynomial in the roots $\{\alpha_i\}$;
see also \cite[Lemma~5.1]{CaGol}.

On the other hand, if $v$ is archimedean and
$|x|_v > \sum_{j=0}^{d-1} | a_j/a_d |_v^{1/(d-j)}$, then
\begin{align*}
|a_d x^d|_v & =  |x|_v \cdot |a_d x^{d-1}|_v
>
\sum_{j=0}^{d-1} \Big|\frac{a_j}{a_d}\Big|_v^{1/(d-j)}
\cdot |x^{d-j-1}|_v \cdot  |a_d x^j|_v
\geq
\sum_{j=0}^{d-1} \Big|\frac{a_j}{a_d}\Big|_v \cdot |a_dx^j|_v
\\
&=
\sum_{j=0}^{d-1} |a_j x^j|_v
\geq |a_0 + a_1 x + \cdots + a_{d-1} x^{d-1}|_v,
\end{align*}
and hence $\phi(x)\neq 0$.  Thus,
$A \leq \sum_{j=0}^{d-1} | a_j/a_d |_v^{1/(d-j)}$
if $v$ is archimedean.
\end{remark}

\begin{remark}
Proposition~\ref{prop:goodred} can be proven as a
corollary of Proposition~\ref{prop:lhest}, because
the constants $c_v$ and $C_v$ are both clearly zero
if $\phi$ has good reduction.
\end{remark}

The constants $c_v$ and $C_v$ of Proposition~\ref{prop:lhest}
can sometimes be improved (i.e., made smaller) by changing
coordinates, and perhaps even leaving the original base field K.
The following Proposition
shows how local canonical heights change under scaling; but it
actually applies to any linear fractional coordinate change.

\begin{prop}
\label{prop:coords}
Let $K$ be a field with absolute value $v$,
let $\phi(z)\in K[z]$ be a polynomial
of degree $d\geq 2$, and let $\gamma\in\Cv^{\times}$.
Define $\psi(z) = \gamma \phi(\gamma^{-1}z)\in\Cv[z]$.
Then
$$\hat{\lambda}_{v,\phi}(x)
= \hat{\lambda}_{v,\psi}(\gamma x)
\qquad
\text{for all } x\in \Cv .$$
\end{prop}

\begin{proof}
By exchanging $\phi$ and $\psi$ if necessary, we may assume
that $|\gamma|_v\geq 1$.
For any $x\in\Cv$ and $n\geq 0$, let $y=\phi^n(x)$.
Then
$0 \leq \lambda_v(\gamma y) -\lambda_v(y)\leq \log|\gamma|_v$,
because
$\max\{|y|_v,1\}\leq \max\{ |\gamma y|_v,1\}\leq |\gamma|_v \max\{|y|_v,1\}$.
Thus,
\begin{align*}
\hat{\lambda}_{v,\psi}(\gamma x)
-\hat{\lambda}_{v,\phi}(x)
& = 
\lim_{n\to\infty} d^{-n} [\lambda_v (\psi^n(\gamma x))
- \lambda_v (\phi^n(x))]
\\
& =
\lim_{n\to\infty} d^{-n} [\lambda_v (\gamma\phi^n(x))
- \lambda_v (\phi^n(x))] = 0
\qedhere
\end{align*}
\end{proof}

\begin{cor}
\label{cor:coords}
Let $K$ be a field with absolute value $v$,
let $\phi(z)\in K[z]$ be a polynomial of degree $d\geq 2$,
 and
let $\hat{\lambda}_{v,\phi}$ be the associated
local canonical height.
Let $\gamma\in\Cv^{\times}$, and
define $\psi(z) = \gamma \phi(\gamma^{-1}z)\in\Cv[z]$.
Let $c_v$ and $C_v$ be the constants
from Proposition~\ref{prop:lhest} for $\psi$.
Then for all $x\in\Cv$,
$\dsps
\frac{-d\log c_v}{d-1} \leq
\hat{\lambda}_{v,\phi}(x) - \lambda_v(\gamma x) \leq
\frac{\log C_v}{d-1}$.
\end{cor}

We can now prove the main result of this section.  

\begin{thm}
\label{thm:ghest}
Let $\phi(z)\in \QQ[z]$ be a polynomial of degree $d\geq 2$
with lead coefficient $a\in\QQ^{\times}$.
Let $e\geq 1$ be a positive integer, let
$\gamma=\sqrt[e]{a}\in\Qbar$ be an $e$-th root
of $a$, and define $\psi(z) = \gamma\phi(\gamma^{-1}z)$.
For each $v\in M_{\QQ}$ at which $\phi$ has
bad reduction, let $c_v$ and $C_v$ be the associated constants
in Proposition~\ref{prop:lhest} for $\psi\in\Cv[z]$.  Then
$$
-\frac{1}{d^n} \tilde{c}(\phi,e) \leq
\hat{h}_{\phi}(x) - \frac{1}{ed^n}h\Big( a\big(\phi^n(x)\big)^e \Big)
\leq \frac{1}{d^n} \tilde{C}(\phi,e),$$
for all $x\in \QQ$ and all integers $n\geq 0$,
where
$$\tilde{c}(\phi,e) = 
\frac{d}{d-1} \sum_{v \, \text{\em bad}} \log c_v,
\qquad\text{and}\qquad
\tilde{C}(\phi,e)= 
\frac{1}{d-1} \sum_{v \, \text{\em bad}} \log C_v.
$$
\end{thm}

\begin{proof}
For any prime $v$ of good reduction for $\phi$,
we have $|a|_v = 1$; therefore $|\gamma|_v =1$,
and $\lambda_v(\gamma y) = \lambda_v(y)$ for all $y\in\Cv$.
Hence,
by equation~\eqref{eq:hfeqn},
Propositions~\ref{prop:locglob} and ~\ref{prop:goodred},
and Corollary~\ref{cor:coords},
we compute
\begin{align*}
d^n \hat{h}_{\phi}(x) &=
\hat{h}_{\phi}\big(\phi^n(x)\big) =
\sum_{v\in M_\QQ} \hat{\lambda}_{v,\phi}\big(\phi^n(x)\big)
= 
\sum_{v\text{ good}} \lambda_v\big(\phi^n(x)\big)
+ \sum_{v\text{ bad}} \hat{\lambda}_{v,\phi}\big(\phi^n(x)\big)
\\
&
\geq - \tilde{c}(\phi,e) +
\sum_{v\in M_\QQ} \lambda_v\big(\gamma \phi^n(x)\big)
= - \tilde{c}(\phi,e) +
\frac{1}{e}\sum_{v\in M_\QQ} \lambda_v\Big(a \big(\phi^n(x)\big)^e \Big),
\end{align*}
since
$e\lambda_v(y) = \lambda_v(y^e)$ for all $y\in\Cv$.
The lower bound is now immediate from
the summation formula~\eqref{eq:hsumlog}.
The proof of the upper bound is similar.
\end{proof}

\begin{remark}
The point of Theorem~\ref{thm:ghest} is to approximate
$\hat{h}_{\phi}(x)$ even more accurately than
the naive estimate
$d^{-n}h(\phi^n(x))$, by first changing coordinates
to make $\phi$ monic.  Of course, that coordinate change
may not be defined over $\QQ$; fortunately, the
expression $a(\phi^n(x))^e$ at the heart of the
Theorem still lies in $\QQ$, and hence its height is
easy to compute quickly.
\end{remark}

\begin{remark}
By essentially the same proof, Theorem~\ref{thm:ghest} also holds (with
appropriate modifications) for any global field $K$ in place of $\QQ$.
\end{remark}

\section{Filled Julia sets}
\label{sect:fjsets}

The following definition is standard in both complex
and non-archimedean dynamics.

\begin{defin}
\label{def:fjul}
Let $K$ be a field with absolute value $v$, and
let $\phi(z)\in K[z]$ be a polynomial of degree $d\geq 2$.
The {\em filled Julia set} $\frakK_v$ of $\phi$ at $v$ is
$$\frakK_v := \big\{ x\in \Cv : \{|\phi^n(x)|_v : n\geq 0\}
\text{ is bounded} \big\}.$$
\end{defin}

Note that $\phi^{-1}(\frakK_v)=\frakK_v$.  Also note that
$\frakK_v$ can be 
defined equivalently
as the set of $x\in\Cv$ such that
$|\phi^n(x)|_v\not\to\infty$ as $n\to \infty$.
In addition, the following well known result relates
$\frakK_v$ to $\hat{\lambda}_{v,\phi}$;
the (easy) proof can be found in \cite[Theorem~6.2]{CaGol}.

\begin{prop}
\label{prop:fjzero}
Let $K$ be a field with absolute value $v$, and
let $\phi(z)\in K[z]$ be a polynomial of degree $d\geq 2$.
For any $x\in\Cv$, we have $\hat{\lambda}_{v,\phi}(x)=0$
if and only if $x\in\frakK_v$.
\end{prop}

Because the local canonical height of a polynomial takes on only
nonnegative values,
Propositions~\ref{prop:locglob} and~\ref{prop:fjzero} imply
that any rational preperiodic points must lie in $\frakK_v$
at every place $v$.  However, $\frakK_v$ is often a
complicated fractal set.
Thus, the following Lemmas, which
specify disks containing $\frakK_v$, will be useful.
We set some notation: for any $x\in\Cv$
and $r>0$, we denote the open and closed disks of radius $r$
about $x$ by
$$D(x,r) = \{y\in\Cv : |y-x|_v < r\}
\qquad \text{and} \qquad
\Dbar(x,r) = \{y\in\Cv : |y-x|_v \leq r\}.$$

\begin{lemma}
\label{lem:disk0}
Let $K$ be a field with non-archimedean absolute value $v$,
let $\phi(z)\in K[z]$ be a polynomial of degree $d\geq 2$
and lead coefficient $a_d$, and
let $\frakK_v\subseteq \Cv$ be the filled Julia set of $\phi$.
Define $s_v=\max\{A,|a_d|_v^{-1/(d-1)}\}$, where
$A=\max\{|\alpha|_v : \phi(\alpha)=0\}$
as in Proposition~\ref{prop:lhest}.
Then $\frakK_v \subseteq \Dbar(0,s_v)$.
\end{lemma}

\begin{proof}
See \cite[Lemma~5.1]{CaGol}.  Alternately,
it is easy to check directly that if $|x|_v > s_v$, then
$|\phi(x)|_v = |a_d x^d|_v > |x|_v$; it follows that
$|\phi^n(x)|_v \rightarrow \infty$.
\end{proof}

\begin{lemma}
\label{lem:preim0}
Let $K$ be a field with non-archimedean absolute value $v$,
and
let $\phi(z)\in K[z]$ be a polynomial of degree $d\geq 2$
with lead coefficient $a_d$.
Let $\frakK_v\subseteq\Cv$ be the filled Julia set of $\phi$ at $v$,
let $r_v=\sup\{|x-y|_v:x,y\in\frakK_v\}$ be the diameter
of $\frakK_v$, and let $U_0\subseteq\Cv$ be the intersection
of all disks containing $\frakK_v$.  Then:
\begin{enumerate}
\item $U_0=\Dbar(x,r_v)$ for any $x\in\frakK_v$.
\item There exists $x\in\Cv$ such that $|x|_v=r_v$.
\item $r_v\geq |a_d|_v^{-1/(d-1)}$, with equality if and only if
$\frakK_v=U_0$.
\item If $r_v> |a_d|_v^{-1/(d-1)}$, let $\alpha\in U_0$, and
let $\beta_1,\ldots,\beta_d\in\Cv$ be the roots of $\phi(z)=\alpha$.
Then $\frakK_v \subseteq U_1$, where
$\dsps U_1= \bigcup_{i=1}^{d}\Dbar(\beta_i,|a_d|_v^{-1/{d-1}})$.
\end{enumerate}
\end{lemma}

\begin{proof}
Parts~(1--3) are simply a rephrasing of \cite[Lemma~2.5]{Ben9}.

As for part~(4), if $r_v=|a_d|_v^{-1/(d-1)}$,
then $\frakK_v=U_0$ by part~(3), and hence also
$\phi^{-1}(U_0)=\phi^{-1}(\frakK_v)=\frakK_v=U_0$.
In particular, $\beta_i\in U_0$ for all $i$, and the
result follows.

If $r_v>|a_d|_v^{-1/(d-1)}$,
\cite[Lemma~2.7]{Ben9} says that
$\phi^{-1}(U_0)$ is a disjoint union of $\ell$ strictly smaller disks
$V_1,\ldots,V_\ell$,
each contained in $U_0$, and each of which maps onto $U_0$ under $\phi$,
for some integer $2\leq \ell \leq d$.

Suppose there is some $x\in\frakK_v$ such that
$|x-\beta_i|_v > |a_d|_v^{-1/(d-1)}$ for all $i=1,\ldots,d$.
By part~(1), there is some $y\in\frakK_v$ such that $|x-y|_v=r_v$.
Without loss, $x\in V_1$ and $y\in V_2$;
$V_1$ and $V_2$ are distinct and in fact disjoint,
because each has radius strictly smaller than $r_v$,
and $v$ is non-archimedean.
The disk $V_2$
must also contain some $\beta_j$ (without loss, $\beta_d$),
since $\phi(V_2)=U_0$ by
the previous paragraph;
hence $|x-\beta_d|_v = r_v$.  Thus,
$$|\phi(x)-\alpha|_v =
|a_d|_v \cdot |x-\beta_d|_v \prod_{i=1}^{d-1} |x-\beta_i|_v
> |a_d|_v \cdot r_v \cdot (|a_d|_v^{-1/(d-1)})^{d-1} = r_v.$$
However, $\phi(x)\in \frakK_v \subseteq U_0$ and $\alpha \in U_0$,
and therefore $|\phi(x)-\alpha|_v\leq r_v$.  Contradiction.
\end{proof}

\begin{remark}
Lemma~\ref{lem:preim0}(4) says that
$\frakK_v$ is contained in a union of at most $d$
disks of radius $|a_d|_v^{-1/(d-1)}$.  However, if $d\geq 3$,
then at most one of the disks needs to be that large; the
rest can be strictly smaller.  Still,
the weaker statement of
Lemma~\ref{lem:preim0} above suffices for our purposes.
\end{remark}

\section{Cubic Polynomials}
\label{sect:cubics}

In the study of
quadratic polynomial dynamics, it is useful to note that (except
in characteristic~$2$) any such
polynomial is conjugate over the base field
to a unique one of the form $z^2+c$.  For cubics, it might
appear at first glance that a good corresponding form would be
$z^3+az+b$.  However, this form is not unique,
since
$z^3+az+b$ is conjugate to $z^3+az-b$ by $z\mapsto -z$.
In addition,
it is not even possible to make most cubic
polynomials monic by conjugation over $\QQ$.
More precisely,
if $\phi$ is a cubic with leading coefficient $a$,
and if $\eta(z)=\alpha z + \beta$,
then $\eta^{-1}\circ\phi\circ\eta$ has leading coefficient
$\alpha^{-2}a$, which can only be $1$ if $a$ is a perfect
square.  Instead of $z^3+az+b$, then, we propose
the following two forms as normal forms when conjugating
over a (not necessarily algebraically closed) field of
characteristic not equal to three.
\begin{defin}
\label{def:normalform}
Let $K$ be a field, and let $\phi\in K[z]$ be a cubic polynomial.
We will say that $\phi$ is {\em in normal form} if either
\begin{equation}
\label{eq:normform1}
\phi(z) = a z^3 + bz + 1
\end{equation}
or
\begin{equation}
\label{eq:normform2}
\phi(z) = a z^3 + bz.
\end{equation}
\end{defin}


\begin{prop}
\label{prop:normform}
Let $K$ be a field of characteristic not equal to $3$, and let
$\phi(z)\in K[z]$ be a cubic polynomial.  Then there is a
degree one polynomial $\eta\in K[z]$ such that
$\psi = \eta^{-1}\circ \phi \circ \eta$
is in normal form.
Moreover, if another conjugacy
$\tilde{\eta}(z)$ also gives a normal form
$\tilde{\psi}=\tilde{\eta}^{-1}\circ \phi\circ\tilde{\eta}$,
then either $\tilde{\eta}=\eta$ and $\tilde{\psi}=\psi$,
or else both normal forms $\psi(z)=az^3+bz$
and $\tilde{\psi}(z)=\tilde{a}z^3 + bz$ are of the type
in~\eqref{eq:normform2} with the same linear term,
and the quotient $\tilde{a}/a$ of their
lead coefficients is the square of an element of $K$.
\end{prop}

\begin{proof}
Write $\phi(z) = a z^3 + b z^2 + c z + d \in K[z]$, with $a\neq 0$.
Conjugating by $\eta_1(z)=z - b/(3a)$ gives
$$\psi_1(z) :=\eta_1^{-1}\circ\phi\circ\eta_1(z)
=a z^3 + b' z + d'.$$
(Note that $b',d' \in K$ can be computed explicitly in terms of
$a,b,c,d$, but their precise values are not important here.)
If $d'=0$, then we have a normal form of the type in~\eqref{eq:normform2}.
Otherwise, conjugating $\psi_1$ by $\eta_2(z)=d' z$ gives
the normal form
$$\eta_2^{-1}\circ\psi_1\circ\eta_2(z)
=a'z^3 + b' z + 1,$$
where $a'=a/(d')^2$.

For the uniqueness, suppose
$\phi_1 = \eta^{-1}\circ\phi_2\circ\eta$,
where 
$\eta(z)=\alpha z + \beta$,
$\phi_1(z)=a_1 z^3+b_1 z+c_1$ and
$\phi_2(z)=a_2 z^3+b_2 z+c_2$,
with $c_1,c_2\in\{0,1\}$ and $\alpha a_1 a_2\neq 0$.
Because the $z^2$-coefficient of
$\eta^{-1}\circ\phi_2\circ\eta(z)$
is $\alpha\beta a_1$, we must have $\beta=0$.
Thus, $\phi_2(z)=\alpha^{-1}\phi_1(\alpha z)$,
which means that $c_2=\alpha c_1$ and
$a_2/a_1 = \alpha^2$.  If either $c_1$ or $c_2$ is $1$,
then $\alpha=1$ and $\phi_1=\phi_2$.  Otherwise,
we have $c_1=c_2=0$, $b_1=b_2$,
and $a_2/a_1\in (K^\times)^2$, as claimed.
\end{proof}

\begin{remark}
\label{rem:c0form}
The cubic $\phi(z)=az^3+bz$ is self-conjugate under
$z\mapsto -z$; that is, $\phi(-z)=-\phi(z)$.  (It is not
a coincidence that those cubic polynomials admitting non-trivial
self-conjugacies are precisely those with the more complicated ``$\tilde{a}/a$
is a square'' condition in Proposition~\ref{prop:normform}; see
\cite[Example~4.75 and Theorem~4.79]{Sil}.)  As a result,
$\hat{h}_{\phi}(-x) = \hat{h}_{\phi}(x)$ for all $x\in\QQ$;
and if $x$
is a preperiodic point of $\phi$, then so is $-x$.

In addition, the function $-\phi(z)=-az^3 - bz$
satisfies $(-\phi)\circ(-\phi)=\phi\circ\phi$.  Thus,
$\hat{h}_\phi(x)=\hat{h}_{-\phi}(x)$ for all $x\in\QQ$.
Moreover, $\phi$ and $-\phi$ have the same set of preperiodic
points, albeit with slightly different
arrangements of points into cycles.
\end{remark}

The normal forms of Definition~\ref{def:normalform} have two key
uses.  The first is that they allow us to list a unique (or,
in the case of form~\eqref{eq:normform2}, essentially unique) element
of each conjugacy class of cubic polynomials over $\QQ$ in a systematic
way, which is helpful for having a computer algorithm test them one
at a time.  The second is that the forms provide a description
of the moduli space $\calM_3$
of all cubic polynomials up to conjugation.
This second use is crucial to the very statement of Conjecture~\ref{conj:ht},
because the quantity $h(\phi)$ is defined to be the height of the
conjugacy class of $\phi$ viewed as a point on $\calM_3$.

In particular, Proposition~\ref{prop:normform} says that $\calM_3$
can be partitioned into two pieces: the first piece is an affine
subvariety of $\PP^2$,
and the second is an affine line.
More specifically,
the conjugacy class of the polynomial $\phi(z)=az^3 +bz+1$
corresponds to the point $(a,b)$ in
$\{(a,b)\in\AAA^2 : a\neq 0\}$.
To compute heights, then, we should view $\AAA^2$ as an affine subvariety
of $\PP^2$, thus declaring $h(\phi)$ to be the height $h([a:b:1])$
of the point $[a:b:1]$ in $\PP^2$.
Meanwhile, the conjugacy class
of $az^3 + bz$ over $\Qbar$ is determined solely by $b$,
because $az^3 + bz$ is conjugate to $a'z^3+bz$
over $\Qbar$.
(As noted in \cite[Section~4.4 and Remark~4.39]{Sil},
$\calM_3$ is the moduli space of $\Qbar$-conjugacy classes
of cubic polynomials, not $\QQ$-conjugacy classes.)
Thus, the $\Qbar$-conjugacy class of
$\phi_0(z)=az^3+bz$ corresponds to the point $b$ in $\AAA^1$,
and the corresponding height is $h(\phi_0)=h([b:1])$,
the height of the point $[b:1]$ in $\PP^1$.
We phrase these assignments formally in the following definition.

\begin{defin}
\label{def:cubht}
Given $a,b\in\QQ$ with $a\neq 0$, define $\phi(z)=az^3 + bz+1$
and $\phi_0(z)=az^3 + bz$.
Write $a=k/m$ and $b=\ell/m$
with $k,\ell,m\in\ZZ$ and $\gcd(k,\ell,m)=1$;
also write $b=\ell_0/m_0$ with $\gcd(\ell_0,m_0)=1$.
Then we define
the {\em heights} $h(\phi)$, $h(\phi_0)$ of the maps $\phi$
and $\phi_0$ to be
$$h(\phi):=\log\max\{|k|_{\infty},|\ell|_{\infty},|m|_{\infty}\}
\qquad\text{and}\qquad
h(\phi_0):=\log\max\{|\ell_0|_{\infty},|m_0|_{\infty}\}.$$
\end{defin}

Note that $h(\phi_0)=h(b)=\sum_v \log\max\{1,|b|_v\}$,
and $h(\phi) = \sum_v \log\max\{1,|a|_v,|b|_v\}$.


\begin{prop}
\label{prop:normbd}
Given $a,b,\phi,\phi_0$ as in Definition~\ref{def:cubht},
let $\gamma=\sqrt{a}\in\Qbar$ be a square root of $a$,
and define
$$\psi(z) = \gamma \phi(\gamma^{-1} z)
= z^3 + bz + \sqrt{a},
\qquad
\text{and}
\qquad
\psi_0(z) = \gamma \phi_0(\gamma^{-1} z)
= z^3 + bz.$$
Let $\tilde{c}(\phi,2)$, $\tilde{C}(\phi,2)$,
$\tilde{c}(\phi_0,2)$, and $\tilde{C}(\phi_0,2)$
be the corresponding constants from
Theorem~\ref{thm:ghest}.
Then
\begin{align*}
\tilde{c}(\phi,2) & \leq 1.84\cdot \max\{h(\phi),1\}, &
\tilde{C}(\phi,2) & \leq .75\cdot \max\{h(\phi),1\},
\\
\tilde{c}(\phi_0,2) & \leq 1.57\cdot \max\{h(\phi_0),1\},
\qquad\text{and}
&
\tilde{C}(\phi_0,2) &\leq .75\cdot h(\phi_0).
\end{align*}
\end{prop}

\begin{proof}
Note that
$$\log(1 + |a|_{\infty}^{1/6} + |b|_{\infty}^{1/2})
\leq \log\big( 3 \max\{1,|a|_{\infty}^{1/6},|b|_{\infty}^{1/2} \} \big)
\leq \log 3 + \frac{1}{2} \log \max\{1,|a|_{\infty},|b|_{\infty}\}.$$
Thus,
if $h(\phi)\geq \log 9$, then
by Remark~\ref{rem:Adef}
and the definition of $\tilde{c}(\phi,2)$,
\begin{align*}
\frac{2}{3}\tilde{c}(\phi,2)
& \leq
\log(1 + |a|_{\infty}^{1/6} + |b|_{\infty}^{1/2})
+ \sum_{v\neq\infty} \log\max\{1,|a|_v^{1/6},|b|_v^{1/2}\}
\\
& \leq
\log 3 + \frac{1}{2}\sum_v \log\max\{1,|a|_v,|b|_v\}
= 
\log 3 + \frac{1}{2}h(\phi)
\leq h(\phi).
\end{align*}
Hence, $\tilde{c}(\phi,2)<1.5\cdot h(\phi)$ if
$h(\phi)> \log 9$.

Similarly,
$\log(1 + |a|_{\infty} + |b|_{\infty})
\leq \log 3 + \log\max\{1,|a|_{\infty},|b|_{\infty}\}$,
and therefore
$$2\tilde{C}(\phi,2) \leq
\log 3 + h(\phi) \leq 1.5\cdot h(\phi)$$
if $h(\phi)\geq \log 9$.
The bounds for $\phi_0$ can be proven in the same fashion
in the case that $h(\phi_0)\geq \log 4$.  (That is, $h(b)\geq 4$.)

Finally, there are 
fifteen choices of $b\in\QQ$ for which $h(b)<\log 4$,
and 1842 pairs $(a,b)\in\QQ$ for which
$h(\phi)<\log 9$.  By a simple computer computation
(working directly from the definitions in
Proposition~\ref{prop:lhest}
and Theorem~\ref{thm:ghest}, not the estimates 
of Remark~\ref{rem:Adef}), one can
check that the desired inequalities hold in all cases.
\end{proof}

\begin{remark}
In fact, 
$\tilde{c}(\phi_0,2)\leq 1.5\cdot \max\{h(b),1\}$
in all but four cases: $b=\pm 2/3$ and $b=\pm 3/2$,
which give $h(b)=\log 3$ and 
$\tilde{c}(\phi_0,2)=1.5\cdot \log(\sqrt{2} + \sqrt{3})$.
Similarly,
$\tilde{c}(\phi,2)\leq 1.5 \cdot\max\{h(\phi),1\}$
in all but 80 cases.
The maximum ratio of $1.838\ldots$
is attained twice, when $(a,b)$ is
$(-1,2/3)$ or $(1,-2/3)$.  In both cases, $h(\phi)=\log 3$
and 
$\tilde{c}(\phi_0,2)=1.5\cdot \log((\alpha+1)\sqrt{3})$,
where $\alpha\approx 1.22$ is the unique real root
of $3z^3 - 2z - 3$.
\end{remark}

The next Lemma says that for cubic polynomials in normal form,
and for $v$ a $p$-adic absolute value with $p\neq 3$,
the radius $s_v$ from Lemma~\ref{lem:disk0}
coincides with the radius $r_v$ from Lemma~\ref{lem:preim0}.
Thus, when we search for rational preperiodic points, we
are losing no efficiency by searching in $\Dbar(0,s_v)$
instead of the ostensibly smaller disk $U_0$.

\begin{lemma}
\label{lem:kdiam}
Let $K$ be a field
with non-archimedean absolute value $v$
such that $|3|_v=1$.
Let $\phi(z)\in K[z]$ be a cubic polynomial in normal form,
and let $\frakK_v\subseteq\Cv$ be the filled Julia set of $\phi$
at $v$.  Let $r_v=\sup\{|x-y|_v:x,y\in\frakK_v\}$ be the diameter
of $\frakK_v$.  Then $|x|_v\leq r_v$ for all $x\in\frakK_v$.
\end{lemma}

\begin{proof}
If $\phi(z)=az^3 + bz$, then $\phi(0)=0$, and therefore
$0\in\frakK_v$.  The desired conclusion is immediate.
Thus, we consider $\phi(z)=az^3 + bz + 1$.  
Note that the three
roots $\alpha,\beta,\gamma\in\Cv$
of the equation $\phi(z) - z=0$ are
fixed by $\phi$ and hence lie in $\frakK_v$.

Without loss, assume $|\alpha|_v \geq |\beta|_v \geq |\gamma|_v$.
It suffices to show that $|\alpha-\gamma|_v = |\alpha|_v$;
if $x\in\frakK_v$, then
$|x|_v \leq \max\{|x-\alpha|_v,|\alpha-\gamma|_v\}\leq r_v$,
as desired.
If $|\alpha|_v > |\gamma|_v$, then $|\alpha-\gamma|_v=|\alpha|_v$,
and we are done.  Thus, we may assume
$|\alpha|_v=|\beta|_v=|\gamma|_v=|a|_v^{-1/3}$, which
implies that $|a|_v \geq |b-1|_v^3$.
We may also assume that $|\alpha-\gamma|_v \geq |\alpha-\beta|_v$.

The polynomial $Q(z)=\phi(z+\alpha)-(z+\alpha)$ has roots
$0$, $\beta-\alpha$, and $\gamma - \alpha$; on the other hand,
$Q(z) = az[z^2 + 3\alpha z + (a^{-1}(b-1) + 3\alpha^2)]$
by direct computation.  Thus,
\begin{equation}
\label{eq:qzdef}
\big( z - (\beta-\alpha) \big) \big( z - (\gamma-\alpha) \big)
= z^2 + 3\alpha z + \big(a^{-1}(b-1) + 3\alpha^2\big).
\end{equation}
Since $|a^{-1}(b-1)|_v\leq |a|_v^{-2/3} = |\alpha|_v^2$, the
constant term of \eqref{eq:qzdef} has absolute value at most
$|\alpha|_v^2$; meanwhile, the linear coefficient satisfies
$|3\alpha|_v=|\alpha|_v$.
Thus, either from the Newton polygon
or simply by inspection of \eqref{eq:qzdef}, it follows that
$|\alpha-\gamma|_v=|\alpha|_v$.
\end{proof}

%

\begin{remark}
Lemma~\ref{lem:kdiam} can be false in non-archimedean fields
in which $|3|_v<1$.  For example, if $K=\QQ_3$
(in which $|3|_3=1/3<1$) and
$\phi(z) = -(1/27) z^3 + z + 1$, then it is not difficult
to show that the diameter of
the filled Julia set is $3^{-3/2}$.  However, $\alpha=3$
is a fixed point, and $|\alpha|_3=1/3 > 3^{-3/2}$.
\end{remark}

At the archimedean place $v=\infty$, we will study
not $\frakK_{\infty}$ itself, but rather the simpler
set $\frakK_{\infty}\cap\RR$, which we will describe
in terms of the real fixed points.
Note, of course, that any cubic with real coefficients has at
least one real fixed point; and if there are exactly two real
fixed points, then one must appear with multiplicity two.

\begin{lemma}
\label{lem:onepos}
Let $\phi(z) \in \RR[z]$ be a cubic polynomial with
positive lead coefficient.
If $\phi$ has precisely one real fixed point $\gamma\in\mathbb{R}$,
then $\frakK_{\infty}\cap\RR=\{\gamma\}$ is a single point.
\end{lemma}

\begin{proof}
We can write $\phi(z) = z + (z-\gamma)^j \psi(z)$,
where $1\leq j\leq 3$, and
where $\psi\in\RR[z]$ has positive lead coefficient and no real roots.
Thus, there is a
positive constant $c>0$ such that $\psi(x)\geq c$ for all $x\in\RR$.
Given any $x\in\RR$ with $x>\gamma$, then, $\phi(x)> x + c(x-\gamma)^j$.
It follows that $\phi^n(x) > x + nc(x-\gamma)^j$, and hence
$\phi^n(x)\rightarrow \infty$ as $n\rightarrow\infty$.
Similarly,  for $x<\gamma$,
$\phi^n(x)\rightarrow -\infty$ as $n\rightarrow\infty$.
\end{proof}

\begin{lemma}
\label{lem:apos}
Let $\phi(z)\in\RR[z]$ be a cubic polynomial with
positive lead coefficient $a>0$ and
at least two distinct fixed points.
Denote the fixed points by
$\gamma_1, \gamma_2, \gamma_3\in\RR$, with
$\gamma_1 \leq \gamma_2 \leq \gamma_3$.
Then $\frakK_{\infty}\cap\RR \subseteq [\gamma_1,\gamma_3]$,
and
$$
\phi^{-1}([\gamma_1,\gamma_3]) \subseteq
[\gamma_1, \gamma_1 + a^{-1/2}] \cup
[\gamma_2 - a^{-1/2} , \gamma_2 + a^{-1/2}]
\cup [\gamma_3 - a^{-1/2}, \gamma_3] .$$
\end{lemma}

\begin{proof}
Let $\alpha = \inf (\frakK_{\infty}\cap\RR )$;
then $\alpha\in \frakK_{\infty}\cap\RR$, since this set is closed.
Therefore, $\phi(\alpha)\geq \alpha$,
because $\phi(\frakK_{\infty}\cap\RR)\subseteq\frakK_{\infty}\cap\RR$.
On the other hand, if $\phi(\alpha)>\alpha$, then
by continuity (and because $\phi$ has positive lead coefficient),
there is some $\alpha' < \alpha$ such that $\phi(\alpha') = \alpha$,
contradicting the minimality of $\alpha$.
Thus $\phi(\alpha) = \alpha$, giving $\alpha = \gamma_1$. Similarly, 
$\sup (\frakK_{\infty}) = \gamma_3$, proving the first
statement.

For the second statement, note that
$\phi(z) = a(z-\gamma_1)(z-\gamma_2)(z-\gamma_3) + z$,
and consider $x\in\RR$ outside all
three desired intervals.  We will show
that $\phi(x)\not\in [\gamma_1,\gamma_3]$.

If $x>\gamma_3$,
then $\phi(x) > x > \gamma_3$.  Similarly, if $x<\gamma_1$,
then $\phi(x) < x < \gamma_1$.

If $\gamma_1 + a^{-1/2} < x <  \gamma_2 - a^{-1/2}$, 
then, noting that $\gamma_3>x$, we have
$$\phi(x) - \gamma_3 = \big[a(x-\gamma_1)(\gamma_2-x) - 1\big](\gamma_3 -x)
> (a(a^{-1/2})^2 - 1) (\gamma_3-x) \geq  0.$$
Similarly, 
if $\gamma_2 + a^{-1/2} < x <  \gamma_3 - a^{-1/2}$,
we obtain $\phi(x) < \gamma_1$.
\end{proof}

\begin{lemma}
\label{lem:aneg}
\label{lem:oneneg}
Let $\phi(z)\in\RR[z]$ be a cubic polynomial with
negative lead coefficient.  If $\frakK_{\infty}\cap\RR$
consists of more than one point,
then $\phi$ has at least
two distinct real periodic points of period two.
Moreover, if $\alpha\in\RR$ is the smallest such periodic
point, then $\phi(\alpha)$ is the largest,
and $\frakK_{\infty}\cap\RR \subseteq [\alpha,\phi(\alpha)]$.
\end{lemma}

\begin{proof}
Let $\alpha = \inf(\frakK_{\infty}\cap\RR)$
and $\beta = \sup(\frakK_{\infty}\cap\RR)$,
so that $\frakK_{\infty}\cap\RR \subseteq [\alpha,\beta]$.
By hypothesis, $\alpha<\beta$.  It suffices to show
that $\phi(\alpha)=\beta$ and $\phi(\beta)=\alpha$.

Note that $\phi(\alpha)\in\frakK_{\infty}\cap\RR$,
and therefore $\phi(\alpha)\leq\beta$.
If $\phi(\alpha) < \beta$, then by continuity, there is
some $\alpha' < \alpha$ such that $\phi(\alpha') = \beta$,
contradicting the minimality of $\alpha$.
Thus, $\phi(\alpha)=\beta$; similarly,
$\phi(\beta)=\alpha$.
\end{proof}

\section{The search algorithm}
\label{sect:alg}

We are now
ready to describe our algorithm to search for
preperiodic points and points of small height for cubic
polynomials over $\QQ$.

\begin{alg}
\label{alg:main}
Given $a\in\QQ^{\times}$ and $b\in\QQ$,
set $\phi(z)=az^3 + bz + 1$ or $\phi(z)=az^3 + bz$,
define $h(\phi)$ as in Definition~\ref{def:cubht},
and set $h_+(\phi)=\max\{h(\phi),1\}$.

1.
Let $S$ be the set of all (bad) prime 
factors $p$ of the numerator of $a$, denominator of $a$,
and denominator of $b$.
Compute each radius $s_p$ from Lemma~\ref{lem:disk0};
by Remark~\ref{rem:Adef},
$$
s_p = \begin{cases}
\max\{|b/a|_p^{1/2}, |1/a|_p^{1/2} \} &
\text{ for } \phi(z) = az^3 + bz,
\\
\max\{|b/a|_p^{1/2}, |1/a|_p^{1/3}, |1/a|_p^{1/2} \} &
\text{ for } \phi(z) = az^3 + bz + 1.
\end{cases}
$$
Shrink $s_p$ if necessary to be an integer power of $p$.
Let $M = \prod_{p\in S} s_p \in\QQ^{\times}$.
Thus, for any preperiodic rational point $x\in\QQ$,
we have $Mx\in\ZZ$.

2. If $a>0$ and $\phi$ has only one real fixed point, or
if $a<0$ and $\phi$ has no real two-periodic points,
then (by Lemma~\ref{lem:onepos} or Lemma~\ref{lem:oneneg})
$\frakK_\infty\cap\RR$ consists of a single point $\gamma\in\RR$,
which must be fixed.
In that case, check whether $\gamma$ is rational by seeing
whether $M\gamma$ is an integer;
report either the one or zero preperiodic points, and end.

3. Let $S'$ be the set of all $p\in S$ for which  $|a|_p^{-1/2}< s_p$.
Motivated by Lemma~\ref{lem:preim0}(4),
for each such $p$
consider the (zero, one, two, or three) disks of radius
$|a|_p^{-1/2}$
that contain both an element of $\phi^{-1}(0)$
and a $\QQ_p$-rational point.  If for at least one $p\in S'$ there are
no such disks, then report zero preperiodic points, and end.

4. Otherwise, use
the Chinese Remainder Theorem to list all rational numbers
that lie in the real interval(s) given by
Lemma~\ref{lem:apos} or~\ref{lem:aneg}, are integer multiples of
the rational number $M$ from Step~1, and lie in the disks
from Step~3 at each $p\in S'$.

5. For each point $x$ in Step~4, compute
$\phi^i(x)$ for $i=0,\ldots,6$.
If any are repeats, record a preperiodic point.
Otherwise, compute
$h(a(\phi^6(x))^2)/(2\cdot 3^6 \cdot h_+(\phi))$.
If the value is less than $.03$,
record $h(a(\phi^{12}(x))^2)/(2\cdot 3^{12})$
as $\hat{h}_{\phi}(x)$,
and
\begin{equation}
\label{eq:Ndef}
\frakh(x) = \hat{h}_{\phi}(x)/h_+(\phi)
\end{equation}
as the {\em scaled height} of $x$.
\end{alg}

\begin{remark}
The definition of $h_{+}(\phi)$ is designed to avoid dividing
by zero when computing $\frakh(x)$.  In particular,
the choice of $1$ as a minimum value is arbitrary.
Of course, the height $h(\phi)$ already
depends on our choice of normal forms in
Definition~\ref{def:normalform}; moreover, without reference
to some kind of
canonical structure, Weil heights on varieties are only
natural objects up to bounded differences.
In other words,
$h_{+}(\phi)$ is no more arbitrary than  $h(\phi)$
as a height on the moduli space $\calM_3$.

In addition, none of the polynomials we found with
points of particularly small scaled height $\frakh(x)$ had
$h(\phi)\leq 1$, even though the change from
$h(\phi)$ to $h_{+}(\phi)$ could only make $\frakh(x)$
smaller.  Thus, our use of $h_{+}$ had no significant
effect on the data.
%
\end{remark}

\begin{remark}
Algorithm~\ref{alg:main} tests only points that, at all places, are
in regions where the filled Julia set might be.
At non-archimedean places,
that means the region $U_1$ in Lemma~\ref{lem:preim0}(4);
and at the archimedean place, that means the regions
described in Lemma~\ref{lem:apos} or Lemma~\ref{lem:aneg}.
Thus, as mentioned in the discussion
following Proposition~\ref{prop:fjzero}, the algorithm
is guaranteed to 
test all preperiodic points,
but there is a possibility it may miss a point of small positive
height that happens to lie outside the search region
at some place.
However, such a point must have a non-negligible
positive contribution
to its canonical height, coming from the local canonical
height at that place.

For example, 
any point $x$ lying outside the region $U_1$
at a non-archimedean place $v$
must satisfy $\phi(x)\not\in U_0$.
If $p_v\neq 3$, then by Lemma~\ref{lem:kdiam},
$U_0=\Dbar(0,s_v)$, outside of which it is easy to show that
$\hat\lambda_{\phi,v}(x)=\lambda_v(x) + \frac{1}{2}\log |a|_v$.
Since $\hat\lambda_{\phi,v}(\phi(x))>0$
and $\lambda_v(x)$ takes values in $(\log p_v)\ZZ$, it follows that
$\hat\lambda_{\phi,v}(\phi(x))\geq (\log p_v)/2$,
and therefore
$$\hat{h}_\phi(x)\geq \hat\lambda_{\phi,v}(x)\geq \frac{\log p_v}{6}
\geq \frac{\log 2}{6} = .1155\ldots .$$
Thus, we are not missing points of height smaller than $.11$ by
restricting to $U_1$.

Admittedly, at the archimedean place we have no such lower
bound, and the possibility exists of missing a point of small height
just outside the search region.
However, because the denominators of such
points (and all their forward iterates!) must divide $M$,
there still cannot
be many omitted points of small height unless $h(\phi)$
is very large.
\end{remark}

\begin{remark}
\label{rem:arbitrary}
The bounds of $6$ (for preperiodic repeats) and $.03$ (for
$\frakh(x)$), and the decision to test
$\hat{h}_{\phi}(x)$ first at $6$ iterations
and then again at $12$ were chosen by trial and error.
For example, there seem to be many cubic polynomials
with points of scaled height smaller than $.03$,
suggesting that our choice of that cutoff is safely large.


Meanwhile, if there
happened to be a preperiodic chain of length $7$ or longer,
our algorithm would not identify the starting point as preperiodic.
However, the first point in
such a chain would still have shown up in our data
as a point of extraordinarily small
scaled height; but we found no such points in our entire search.
That is, {\em none} of the maps we tested have preperiodic chains
of length greater than six.

Finally, by Proposition~\ref{prop:normbd} and Theorem~\ref{thm:ghest},
our preliminary
estimate (after six iterations) for $\frakh$
is accurate to within $3^{-6}\cdot 1.84 <.0026$,
and our sharper estimate (after twelve iterations)
is accurate
to within $3^{-12}\cdot 1.84 < .0000035$.  Thus, 
the points we test with $\frakh<.027$ or $\frakh>.033$
cannot be misclassified; and our recorded
computations of $\frakh$ are accurate
to at least the first
five places after the decimal point.
\end{remark}

\section{Data Collected}
\label{sect:data}

\begin{table}
\scalebox{0.8}{
\begin{tabular}{|c||c|c||c|c|c|}
\hline
$n$ &
\multicolumn{2}{c ||}{
\parbox{3cm}{
\makebox[3cm][c]{number of form}
\\
\makebox[3cm][c]{$az^3 + bz$}
}
}
&
\multicolumn{2}{c |}{
\parbox{3cm}{
\makebox[3cm][c]{number of form}
\\
\makebox[3cm][c]{$az^3 + bz+1$}
}
}
\\
\hline
$h(a),h(b)<$ & $\log 200$ & $\log 300$ & $\log 200$ & $\log 300$
\\
\hline
11 & 10  &  10 & 0 & 0
\\
\hline
10 & 0  &  0 & 0 & 3 
\\
\hline
9 & 30  & 36 & 20 & 28
\\
\hline
8 &  0  & 0 & 36 & 52
\\
\hline
7 & 196  & 318 & 144 & 193 
\\
\hline
6 &  0 & 0 & 257 & 358 
\\
\hline
5 & 524 & 774 & 533 & 751
\\
\hline
4 &  0 & 0 & 1,533 & 2,314
\\
\hline
3 & 132,352 & 297,826 & 52,402 & 115,954
\\
\hline
2 &  0  & 0 & 42,447 & 92,221
\\
\hline
1 &  422,358,932 &  2,072,790,448 & 187,391 & 432,131
\\
\hline
0 &  0 & 0 & 2,362,307,079 &  11,939,398,165
\\
\hline
total & 422,492,044 & 2,073,089,412 & 2,362,591,842 & 11,940,042,170
\\
\hline
\end{tabular}
}
\caption{Number of distinct
cubic polynomials $az^3 + bz$ and $az^3 + bz +1$
with $h(a),h(b)<\log 200, \log 300$ and $n$ rational points of
scaled height smaller than $.03$.}
\label{tab:numsmall}
\end{table}

We ran Algorithm~\ref{alg:main} on every cubic polynomial
$az^3 + bz +1$ and $az^3 +bz$ for which
$a\in\QQ^{\times}$, $b\in\QQ$, and
both numerators and both denominators are smaller than $300$
in absolute value.  That means $109,270$ choices for $a$
and (because $b=0$ is allowed) $109,271$ choices for $b$,
giving almost 12~billion pairs $(a,b)$.
(Not coincidentally, $109,271$ is approximately
$(12/\pi^2) \cdot 300^2$; see \cite[Exercise~3.2(b)]{Sil}.)
Of course,
in light of Proposition~\ref{prop:normform},
we skipped polynomials
of the form $\gamma^2 a z^3 + bz$ for $\gamma\in \QQ$
if we had already tested $az^3 + bz$.
That meant only $18,972$ choices for $a$, but the same
$109,271$ choices for $b$;
as a result, there were only about $2$~billion truly different 
cubics of the second type.
Combining the two families, then,
we tested over 14~billion truly
different cubic polynomials.
We summarize our key observations here;
the complete data may be found online at
{\tt http://www.cs.amherst.edu/\textasciitilde rlb/cubicdata/}


Table~\ref{tab:numsmall} lists
the number of such polynomials
with a prescribed number of points $x\in\QQ$
of small height (that is, with
$\frakh< .03$, where
$\frakh(x)=\hat{h}_{\phi}(x)/\max\{h(\phi),1\}$
is the scaled height of equation~\eqref{eq:Ndef});
it also lists the totals for $h(a),h(b)<\log 200$,
for comparison.
Of course, every polynomial
of the form $az^3 + bz$ has an odd number of
small height points,
by Remark~\ref{rem:c0form} and because $x=0$ is fixed.
Meanwhile, there are more polynomials
of the form $az^3 + bz+1$ with three small
points than with two, because there are
several ways to have three preperiodic points
(three fixed points, a fixed point with
two extra pre-images, or a $3$-cycle),
but essentially only one way to have two: a $2$-cycle.  
After all, a cubic $\phi$ with two rational fixed points has a
third, except in the rare case of multiple roots
of $\phi(z)-z$; and if $\phi$ has a fixed point
$\alpha\in\QQ$ with a distinct preimage $\beta\in\QQ$, then
the third preimage is also rational.

\begin{table}
\scalebox{0.8}{
\begin{tabular}{|c|c|c|c|}
\hline
$a,b$ & periodic cycles &  
strictly preperiodic
&
\parbox{3.5cm}{
\hfil small  height $>0$
\\
\makebox[3.5cm][s]{point(s) \hfil  $\frakh$}
}
\\
\hline
$-\frac{3}{2},\frac{19}{6}, 0 \vphantom{\sum_{X_X}^X}$ &
$\{\frac{5}{3}, -\frac{5}{3}\}$, $\{0\}$ &
$\pm \frac{4}{3}$, $\pm\frac{2}{3}$,
$\pm\frac{1}{3}$, $\pm 1$
&
---
\\
\hline
$\frac{3}{2},-\frac{19}{6}, 0 \vphantom{\sum_{X_X}^X}$ &
$\{\frac{5}{3}\}$, $\{-\frac{5}{3}\}$, $\{0\}$ &
$\pm \frac{4}{3}$, $\pm\frac{2}{3}$,
$\pm\frac{1}{3}$, $\pm 1$
&
---
\\
\hline
$-3,\frac{37}{12}, 0 \vphantom{\sum_{X_X}^X}$ &
$\{\frac{7}{6},-\frac{7}{6}\}$,
$\{\frac{5}{6}\}$, $\{-\frac{5}{6}\}$, $\{0\}$ &
$\pm \frac{1}{6}$, $\pm \frac{1}{2}$, $\pm \frac{2}{3}$, 
&
---
\\
\hline
$3,-\frac{37}{12}, 0 \vphantom{\sum_{X_X}^X}$ &
$\{\frac{5}{6},-\frac{5}{6}\}$,
$\{\frac{7}{6}\}$, $\{-\frac{7}{6}\}$, $\{0\}$ &
$\pm \frac{1}{6}$, $\pm \frac{1}{2}$, $\pm \frac{2}{3}$, 
&
---
\\
\hline
$-\frac{3}{2},\frac{73}{24}, 0 \vphantom{\sum_{X_X}^X}$ &
$\{\frac{1}{2}, \frac{4}{3}\}$, $\{-\frac{1}{2}, -\frac{4}{3}\}$,
$\{\frac{7}{6}\}$, $\{-\frac{7}{6}\}$,  $\{0\}$ &
$\pm \frac{1}{6}, \pm\frac{3}{2}$
&
---
\\
\hline
$\frac{3}{2},-\frac{73}{24}, 0 \vphantom{\sum_{X_X}^X}$ &
$\{\frac{1}{2}, -\frac{4}{3}\}$, $\{-\frac{1}{2}, \frac{4}{3}\}$,
$\{\frac{7}{6},-\frac{7}{6}\}$, $\{0\}$ &
$\pm \frac{1}{6}, \pm\frac{3}{2}$
&
---
\\
\hline
$-\frac{5}{3},\frac{109}{60}, 0 \vphantom{\sum_{X_X}^X}$ &
$\{\frac{13}{10}, -\frac{13}{10}\}$,
$\{\frac{7}{10}\}$, $-\{\frac{7}{10}\}$, $\{0\}$ &
$\pm \frac{3}{10}$, $\pm\frac{6}{5}$, $\pm\frac{1}{2}$
&
---
\\
\hline
$\frac{5}{3},-\frac{109}{60}, 0 \vphantom{\sum_{X_X}^X}$ &
$\{\frac{7}{10},-\frac{7}{10}\}$,
$\{\frac{13}{10}\}$ , $\{-\frac{13}{10}\}$, $\{0\}$ &
$\pm \frac{3}{10}$, $\pm\frac{6}{5}$, $\pm\frac{1}{2}$
&
---
\\
\hline
$-\frac{6}{5},\frac{169}{120}, 0 \vphantom{\sum_{X_X}^X}$ &
$\{\frac{17}{12}, -\frac{17}{12}\}$,
$\{\frac{7}{12}\}$, $\{-\frac{7}{12}\}$, $\{0\}$ &
$\pm\frac{13}{12}$, $\pm\frac{2}{3}$, $\pm\frac{5}{4}$
&
---
\\
\hline
$\frac{6}{5},-\frac{169}{120}, 0 \vphantom{\sum_{X_X}^X}$ &
$\{\frac{7}{12}, -\frac{7}{12}\}$,
$\{\frac{17}{12}\}$, $\{-\frac{17}{12}\}$, $\{0\}$ &
$\pm\frac{13}{12}$, $\pm\frac{2}{3}$, $\pm\frac{5}{4}$
&
---
\\
\hline
\hline
$\frac{1}{240},-\frac{151}{60}, 1 \vphantom{\sum_{X_X}^X}$ &
$\{-10,22\}, \{12, -22\}, \{18, -20\}$
&
$10$, $-18$, $\pm 28$
&
---
\\
\hline
$-\frac{1}{240},\frac{151}{60}, 1 \vphantom{\sum_{X_X}^X}$ &
---
&
---
&
\parbox{4.4cm}{
\makebox[4.4cm][s]{$-12,20$:  \hfil  $.00244$}
\\
\makebox[4.4cm][s]{$10$, $18$, $-22$, $-28$:  \hfil  $.00733$}
\\
\makebox[4.4cm][s]{$-10$, $-18$, $22$, $28$:  \hfil  $.02198$}
}
\\
\hline
$-\frac{169}{240},\frac{259}{60}, 1 \vphantom{\sum_{X_X}^X}$ &
$\{-2\}, \{-\frac{4}{13}\}, \{\frac{30}{13}\}$
&
$\frac{4}{13}, -\frac{10}{13},-\frac{30}{13}, \pm\frac{34}{13},\frac{36}{13}$
&
\makebox[4.4cm][s]{$-\frac{14}{13}$:  \hfil  $.02947$}
\\
\hline
\end{tabular}
}
\caption{Cubic polynomials $az^3 + bz + c$
with ten or more points of small height}
\label{tab:11pts}
\end{table}

According to our data,
no cubic polynomial with $h(a),h(b)< \log(300)$
has more than $11$ rational points of small height.
In fact, there only ten such polynomials with $11$
small points; see Table~\ref{tab:11pts}.
All ten have $11$ preperiodic points and no other points of
small height; all have $h(a),h(b)<200$;
and all are in the $az^3 + bz$ family.
(Five are negatives of the other five, and
similarly the negative of any preperiodic point is also
preperiodic, as discussed in Remark~\ref{rem:c0form}.)

Table~\ref{tab:11pts} also lists the only three
polynomials in our search with exactly ten points
of small height;
a complete list (ordered by $h(\phi)$)
of those with exactly nine points
of small height can be found in
Tables~\ref{tab:9pts0} and~\ref{tab:9pts1}.
To save space, Table~\ref{tab:9pts0}
only lists polynomials $az^3 + bz$ with $a>0$;
to obtain those with $a<0$,
simply replace each pair $(a,b)$ by $(-a,-b)$ and adjust
the cycle structure of the periodic points according to
Remark~\ref{rem:c0form}.

\begin{table}
\scalebox{0.8}{
\begin{tabular}{|c|c|c|c|}
\hline
$a,b$ & periodic cycles &  
strictly preperiodic
&
\parbox{3.5cm}{
\hfil small  height $>0$
\\
\makebox[3.5cm][s]{point(s) \hfil  $\frakh$}
}
\\
\hline
$\frac{3}{2}, -\frac{13}{6} \vphantom{\sum_{X_X}^X}$ &
$\{0\}$, $\{\frac{2}{3}, -1\}$, $\{-\frac{2}{3}, 1\}$
&
$\pm\frac{1}{3}$, $\pm\frac{4}{3}$
&
---
\\
\hline
$3, -\frac{13}{12} \vphantom{\sum_{X_X}^X}$ &
$\{0\}$, $\{\frac{1}{6}, -\frac{1}{6}\}$,
$\{\frac{5}{6}\}$, $\{-\frac{5}{6}\}$ 
&
$\pm\frac{1}{2}$, $\pm\frac{2}{3}$
&
---
\\
\hline
$\frac{5}{3}, -\frac{49}{15} \vphantom{\sum_{X_X}^X}$ &
$\{0\}$, $\{\frac{8}{5}\}$, $\{-\frac{8}{5}\}$
&
$\pm 1$, $\pm \frac{3}{5}$, $\pm\frac{7}{5}$
&
---
\\
\hline
$\frac{3}{2}, -\frac{49}{24} \vphantom{\sum_{X_X}^X}$ &
$\{0\}$, $\{\frac{5}{6}, -\frac{5}{6}\}$
&
$\pm \frac{1}{2}$, $\pm \frac{4}{3}$, $\pm\frac{7}{6}$
&
---
\\
\hline
$\frac{5}{2}, -\frac{49}{40} \vphantom{\sum_{X_X}^X}$ &
$\{0\}$, $\{\frac{3}{10}, -\frac{3}{10}\}$
&
$\pm \frac{1}{2}$, $\pm \frac{4}{5}$, $\pm\frac{7}{10}$
&
---
\\
\hline
$\frac{6}{5}, -\frac{61}{30} \vphantom{\sum_{X_X}^X}$ &
$\{0\}$, $\{1, -\frac{5}{6}\}$, $\{-1,\frac{5}{6}\}$
&
$\pm\frac{2}{3}$, $\pm\frac{3}{2}$
&
---
\\
\hline
$\frac{6}{5}, -\frac{79}{30} \vphantom{\sum_{X_X}^X}$ &
$\{0\}$, $\{\frac{7}{6}, -\frac{7}{6}\}$
&
$\pm\frac{1}{2}$, $\pm\frac{5}{3}$
&
\makebox[3.5cm][s]{$\pm\frac{1}{3}$:  \hfil  $.02046$}
\\
\hline
$\frac{6}{5}, -\frac{91}{30} \vphantom{\sum_{X_X}^X}$ &
$\{0\}$, $\{\frac{11}{6}\}$, $\{-\frac{11}{6}\}$
&
$\pm 1$, $\pm\frac{5}{6}$
&
\makebox[3.5cm][s]{$\pm \frac{2}{3}$:  \hfil  $.01982$}
\\
\hline
$\frac{6}{5}, -\frac{139}{30} \vphantom{\sum_{X_X}^X}$ &
$\{0\}$, $\{\frac{13}{6}\}$, $\{-\frac{13}{6}\}$
&
$\pm \frac{1}{2}$, $\pm\frac{5}{3}$
&
\makebox[3.5cm][s]{$\pm 2$:  \hfil  $.02525$}
\\
\hline
$\frac{7}{6}, -\frac{163}{42} \vphantom{\sum_{X_X}^X}$ &
$\{0\}$, $\{\frac{11}{7}, -\frac{11}{7}\}$
&
$\pm 2$, $\pm\frac{3}{7}$, $\pm\frac{4}{7}$
&
---
\\
\hline
$\frac{7}{2}, -\frac{169}{56} \vphantom{\sum_{X_X}^X}$ &
$\{0\}$, $\{\frac{15}{14}\}$, $\{-\frac{15}{14}\}$
&
$\pm \frac{1}{2}$, $\pm\frac{4}{7}$, $\pm\frac{13}{14}$
&
---
\\
\hline
$\frac{5}{3}, -\frac{169}{60} \vphantom{\sum_{X_X}^X}$ &
$\{0\}$, $\{\frac{1}{2}, -\frac{6}{5}\}$, $\{-\frac{1}{2}, \frac{6}{5}\}$
&
$\pm\frac{13}{10}$
&
\makebox[3.5cm][s]{$\pm \frac{3}{10}$:  \hfil  $.02744$}
\\
\hline
$\frac{15}{7}, -\frac{169}{105} \vphantom{\sum_{X_X}^X}$ &
$\{0\}$, $\{\frac{8}{15}, -\frac{8}{15}\}$
&
$\pm 1$, $\pm\frac{7}{15}$, $\pm\frac{13}{15}$
&
---
\\
\hline
$\frac{5}{3}, -\frac{181}{60} \vphantom{\sum_{X_X}^X}$ &
$\{0\}$, $\{\frac{11}{10}, -\frac{11}{10}\}$
&
$\pm \frac{3}{2}$, $\pm\frac{2}{5}$, $\pm\frac{9}{10}$
&
---
\\
\hline
$\frac{7}{3}, -\frac{193}{84} \vphantom{\sum_{X_X}^X}$ &
$\{0\}$, $\{\frac{1}{2}, -\frac{6}{7}\}$, $\{-\frac{1}{2}, \frac{6}{7}\}$
&
$\pm \frac{8}{7}$, $\pm\frac{9}{14}$
&
---
\\
\hline
$\frac{5}{3}, -\frac{229}{60} \vphantom{\sum_{X_X}^X}$ &
$\{0\}, \{\frac{13}{10}, -\frac{13}{10}\},
\{\frac{17}{10}\}, \{-\frac{17}{10}\}$
&
$\pm\frac{6}{5}, \pm\frac{1}{2}$
&
---
\\
\hline
$\frac{5}{3}, -\frac{241}{60} \vphantom{\sum_{X_X}^X}$ &
$\{0\}, \{\frac{3}{2}, -\frac{2}{5}\},
\{-\frac{3}{2}, \frac{2}{5}\}$
&
$\pm\frac{8}{5}, \pm\frac{1}{10}$
&
---
\\
\hline
$\frac{6}{5}, -\frac{289}{120} \vphantom{\sum_{X_X}^X}$ &
$\{0\}, \{\frac{2}{3}, -\frac{5}{4}\},
\{-\frac{2}{3}, \frac{5}{4}\},
\{\frac{13}{12}, -\frac{13}{12}\}$
&
$\pm\frac{17}{12}$
&
---
\\
\hline
\end{tabular}
}
\caption{Cubic polynomials
$az^3+bz$ with $a>0$ and nine points of small height}
\label{tab:9pts0}
\end{table}

\begin{table}
\scalebox{0.8}{
\begin{tabular}{|c|c|c|c|}
\hline
$a,b$ & periodic cycles &  
strictly preperiodic
&
\parbox{4.4cm}{
\hfil small height $> 0$
\\
\makebox[4.4cm][s]{point(s) \hfil  $\frakh$}
}
\\
\hline
$\frac{1}{6},-\frac{13}{6} \vphantom{\sum_{X_X}^X}$ &
$\{3, -1\}$
&
$0$, $1$, $\pm 2$, $-3$, $\pm 4$
&
---
\\
\hline
$-\frac{1}{6},\frac{13}{6} \vphantom{\sum_{X_X}^X}$ &
$\{3\}$, $\{-1\}$, $\{-2\}$
&
$0$, $1$, $2$, $-3$, $\pm 4$
&
---
\\
\hline
$\frac{3}{4},-\frac{49}{12} \vphantom{\sum_{X_X}^X}$ &
$\{1, -\frac{7}{3}\}$
&
$0$, $\pm \frac{1}{3}$, $\frac{4}{3}$, $\frac{5}{3}$,
$\frac{7}{3}$, $-\frac{8}{3}$
&
---
\\
\hline
$\frac{3}{8},-\frac{49}{24} \vphantom{\sum_{X_X}^X}$ &
$\{-3\}$, $\{\frac{1}{3}\}$, $\{\frac{8}{3}\}$
&
$-1$, $-\frac{5}{3}$
&
\makebox[4.4cm][s]{$0$, $-\frac{1}{3}$, $\pm\frac{7}{3}$:  \hfil  $.02309$}
\\
\hline
$\frac{25}{24},-\frac{49}{24} \vphantom{\sum_{X_X}^X}$ &
$\{0, 1\}$
&
$\pm\frac{1}{5}$, $\frac{3}{5}$, $\pm\frac{7}{5}$,
$-\frac{8}{5}$, $-\frac{9}{5}$
&
---
\\
\hline
$\frac{5}{12},-\frac{49}{60} \vphantom{\sum_{X_X}^X}$ &
$\{\frac{3}{5}\}$
&
$0$, $\pm 1$, $-\frac{3}{5}$, $\pm\frac{7}{5}$, $\pm \frac{8}{5}$
&
---
\\
\hline
$-\frac{5}{12},\frac{49}{60} \vphantom{\sum_{X_X}^X}$ &
$\{1,\frac{7}{5}\}$
&
$0$, $-1$, $-\frac{7}{5}$, $\pm\frac{3}{5}$, $\pm \frac{8}{5}$
&
---
\\
\hline
$-\frac{49}{48},\frac{19}{12} \vphantom{\sum_{X_X}^X}$ &
$\{\frac{12}{7}, -\frac{10}{7}\}$, $\{\frac{2}{7}, \frac{10}{7}\}$
&
$-\frac{2}{7}$, $\pm\frac{4}{7}$, $\pm\frac{6}{7}$
&
---
\\
\hline
$\frac{2}{15},-\frac{91}{30} \vphantom{\sum_{X_X}^X}$ &
$\{\frac{5}{2}, -\frac{9}{2}\}$
&
$3$, $\pm 5$, $\pm\frac{1}{2}$, $\frac{9}{2}$, $-\frac{11}{2}$
&
---
\\
\hline
$-\frac{1}{30},\frac{91}{30} \vphantom{\sum_{X_X}^X}$ &
$\{4,11,-10\}$
&
$0$, $1$, $9$, $-5$, $-6$
&
\makebox[4.4cm][s]{$-4$:  \hfil  $.01983$}
\\
\hline
$\frac{1}{48},-\frac{31}{12} \vphantom{\sum_{X_X}^X}$ &
$\{-4, 10\}$, $\{-10, 6\}$
&
$\pm 2$, $-6$, $\pm 12$
&
---
\\
\hline
$-\frac{1}{48},\frac{31}{12} \vphantom{\sum_{X_X}^X}$ &
---
&
---
&
\parbox{4.4cm}{
\makebox[4.4cm][s]{$4$:  \hfil  $.00039$}
\\
\makebox[4.4cm][s]{$2$, $10$, $-12$:  \hfil  $.00118$}
\\
\makebox[4.4cm][s]{$\pm 6$:  \hfil  $.00355$}
\\
\makebox[4.4cm][s]{$-2$, $-10$, $12$:  \hfil  $.01065$}
}
\\
\hline
$\frac{49}{48},-\frac{31}{12} \vphantom{\sum_{X_X}^X}$ &
$\{-2\}$, $\{\frac{2}{7}\}$, $\{\frac{12}{7}\}$
&
$-\frac{2}{7}$, $\frac{4}{7}$, 
$\pm\frac{10}{7}$, $-\frac{12}{7}$
&
\makebox[4.4cm][s]{$\frac{6}{7}$  \hfil  $.02587$}
\\
\hline
$\frac{3}{16},-\frac{43}{12} \vphantom{\sum_{X_X}^X}$ &
$\{-4, \frac{10}{3}\}$
&
$-\frac{2}{3}$, $\frac{14}{3}$
&
\parbox{4.4cm}{
\makebox[4.4cm][s]{$2$:  \hfil  $.00677$}
\\
\makebox[4.4cm][s]{$\frac{4}{3}$:  \hfil  $.01653$}
\\
\makebox[4.4cm][s]{$4$, $\frac{2}{3}$, $-\frac{14}{3}$:  \hfil  $.02030$}
}
\\
\hline
$-\frac{3}{16},\frac{43}{12} \vphantom{\sum_{X_X}^X}$ &
$\{-4, -\frac{4}{3}, -\frac{10}{3}\}$
&
$\frac{14}{3}$, $-\frac{2}{3}$
&
\parbox{4.4cm}{
\makebox[4.4cm][s]{$-2$:  \hfil  $.00488$}
\\
\makebox[4.4cm][s]{$4$, $\frac{2}{3}$, $-\frac{14}{3}$:  \hfil  $.01463$}
}
\\
\hline
$\frac{3}{40},-\frac{241}{120} \vphantom{\sum_{X_X}^X}$ &
$\{\frac{1}{3}\}$
&
$\pm 3, 5, -\frac{16}{3}, -\frac{19}{3}$
&
\makebox[4.4cm][s]{$-\frac{1}{3},-5,\frac{16}{3}$:
\hfil  $.02182$}
\\
\hline
$-\frac{3}{40},\frac{241}{120} \vphantom{\sum_{X_X}^X}$ &
$\{-3\}$
&
---
&
\parbox{4.4cm}{
\makebox[4.4cm][s]{$3,-5,-\frac{1}{3},\frac{16}{3},\frac{19}{3}$:
\hfil  $.00933$}
\\
\makebox[4.4cm][s]{$5,\frac{1}{3},-\frac{16}{3}$:  \hfil  $.02800$}
}
\\
\hline
$\frac{27}{80},-\frac{91}{60} \vphantom{\sum_{X_X}^X}$ &
$\{\frac{4}{3}, -\frac{2}{9}\}$
&
$-2$, $-\frac{4}{3}$, $\pm\frac{10}{9}$, $\frac{20}{9}$,
$\pm\frac{22}{9}$
&
---
\\
\hline
$-\frac{27}{80},\frac{91}{60} \vphantom{\sum_{X_X}^X}$ &
---
&
---
&
\parbox{4.4cm}{
\makebox[4.4cm][s]{$2, \frac{2}{9}, -\frac{20}{9}$:  \hfil  $.00505$}
\\
\makebox[4.4cm][s]{$\pm \frac{4}{3}, \pm \frac{10}{9}, \pm \frac{22}{9}$:
\hfil  $.01516$}
}
\\
\hline
$\frac{121}{80},-\frac{91}{20} \vphantom{\sum_{X_X}^X}$ &
$\{-2\}$, $\{\frac{2}{11}\}$, $\{\frac{20}{11}\}$
&
$-\frac{2}{11}$,
$\frac{10}{11}$,
$\frac{12}{11}$,
$\pm\frac{18}{11}$,
$-\frac{20}{11}$
&
---
\\
\hline
$\frac{1}{240},-\frac{91}{60} \vphantom{\sum_{X_X}^X}$ &
$\{-10, 12\}$
&
$10, -12, \pm 22$
&
\makebox[4.4cm][s]{$2,18,-20$:
\hfil  $.01806$}
\\
\hline
$-\frac{1}{240},\frac{91}{60} \vphantom{\sum_{X_X}^X}$ &
$\{-2\},\{-10\},\{12\}$
&
$10,-12,-18,20,\pm 22$
&
---
\\
\hline
$\frac{169}{96},-\frac{133}{24} \vphantom{\sum_{X_X}^X}$ &
$\{-2\}$, $\{\frac{2}{13}\}$, $\{\frac{24}{13}\}$
&
$-\frac{2}{13}$, $\frac{8}{13}$, $\frac{18}{13}$,
$\pm\frac{22}{13}$, $-\frac{24}{13}$
&
---
\\
\hline
$-\frac{289}{240},\frac{139}{60} \vphantom{\sum_{X_X}^X}$ &
$\{\frac{4}{17}, \frac{26}{17} \},
\{-\frac{26}{17}, \frac{30}{17} \}$
&
$\pm\frac{6}{17}, \pm\frac{20}{17}$
&
\makebox[4.4cm][s]{$-\frac{14}{17}$
\hfil  $.02485$}
\\
\hline
$-\frac{27}{80},\frac{151}{60} \vphantom{\sum_{X_X}^X}$ &
$\{\frac{10}{3}, -\frac{28}{9} \}$
&
$2, \frac{10}{9}$
&
\parbox{4.4cm}{
\makebox[4.4cm][s]{$-2, -\frac{10}{9}, \frac{28}{9}, -\frac{22}{9}$:
\hfil  $.01396$}
\\
\makebox[4.4cm][s]{$\frac{22}{9}$  \hfil  $.01418$}
}
\\
\hline
$\frac{3}{112},-\frac{247}{84} \vphantom{\sum_{X_X}^X}$ &
$\{12\}$
&
$2,-\frac{14}{3},-\frac{22}{3},\frac{28}{3},-\frac{34}{3}$
&
\makebox[4.4cm][s]{$-2,-\frac{28}{3}\frac{34}{3}$:
\hfil  $.02313$}
\\
\hline
$-\frac{3}{112},\frac{247}{84} \vphantom{\sum_{X_X}^X}$ &
---
&
---
&
\parbox{4.4cm}{
\makebox[4.4cm][s]{$-2,-12,\frac{14}{3}$ \hfil}
\makebox[4.4cm][s]{$\frac{22}{3},-\frac{28}{3},\frac{34}{3}$:
\hfil  $.01568$}
\\
\makebox[4.4cm][s]{$2,\frac{28}{3},-\frac{34}{3}$:  \hfil  $.01995$}
}
\\
\hline
$\frac{3}{80},-\frac{259}{60} \vphantom{\sum_{X_X}^X}$ &
$\{-12\}$
&
$\frac{10}{3},\frac{26}{3}$
&
\parbox{4.4cm}{
\makebox[4.4cm][s]{$-\frac{4}{3},-10,\frac{34}{3}$:
\hfil  $.02012$}
\\
\makebox[4.4cm][s]{$\frac{4}{3},10,-\frac{34}{3}$:
\hfil  $.02974$}
}
\\
\hline
\end{tabular}
}
\caption{Cubic polynomials $az^3+bz+1$ with nine
points of small height}
\label{tab:9pts1}
\end{table}

\begin{remark}
\label{rem:lastex}
Most of the points sharing the same canonical height in
Tables~\ref{tab:9pts0} and~\ref{tab:9pts1} do so simply
because one or two iterates later, they coincide.
For example, consider the fourth map in Table~\ref{tab:9pts1},
namely
$\phi(z) = \frac{3}{8}z^3 - \frac{49}{24} z + 1$.  The three
points $0$, $\pm \frac{7}{3}$ all satisfy $\phi(x)=1$, and
hence all three have the same canonical height.
Meanwhile, $\phi(-\frac{1}{3})=\frac{5}{3}\neq 1$,
but $\phi(1)=\phi(\frac{5}{3}) = -\frac{2}{3}$, and hence
$-\frac{1}{3}$ also has the same common canonical height.

The map $\phi(z) = -\frac{27}{80}z^3 + \frac{151}{60}z + 1$,
near the bottom of Table~\ref{tab:9pts1},
is an exception to this trend.
The points $-2$, $-\frac{10}{9}$, and $\frac{28}{9}$ all satisfy
$\phi(x) = -\frac{4}{3}$, but all iterates
of $-\frac{22}{9}$ appear to be distinct
from those of $-2$.
Nonetheless, all four points share the same canonical
height
$\hat{h}_{\phi}(-\frac{22}{9})=\hat{h}_{\phi}(-2) = \frac{1}{18}\log 5
\approx .08941$.  (The scaled height $.01396$ is of course
$.08941$ divided by $h(\phi)=\log(604)$.)
We can compute this explicit value as follows.
The bad primes are
$v=2,3,5,\infty$.  In $\RR$, the iterates of all four
points approach the fixed point at $-1.639$.
At $v=3$,
$\phi$ maps
the set $\{x\in\QQ_3 : |x|_3 \leq 9\}$ into itself,
since
$9\phi(z/9)= \frac{1}{3}(z^3-z) - \frac{27}{80}z^3 + \frac{57}{20}z + 9$
maps $3$-adic integers to $3$-adic integers.
At $v=2$, one can show that $\phi$ maps
$\Dbar(4,\frac{1}{16})$ into $\Dbar(2,\frac{1}{4})$,
$\Dbar(2,\frac{1}{8})$ into $\Dbar(-2,\frac{1}{8})$,
$\Dbar(-2,\frac{1}{16})$ into $\Dbar(4,\frac{1}{16})$,
and 
$\Dbar(6,\frac{1}{16})$ into $\Dbar(4,\frac{1}{16})$;
hence the orbit any point $x\in\QQ_2$ in these disks 
stays in the same disks.
Thus, $\hat{\lambda}_{\phi,\infty}(x)=
\hat{\lambda}_{\phi,3}(x)=
\hat{\lambda}_{\phi,2}(x)=0$
for all four points $x$;
by Propositions~\ref{prop:locglob} and~\ref{prop:goodred}, then,
$\hat{h}_{\phi}(x) = \hat{\lambda}_{\phi,5}(x)$.
Finally, all four points
satisfy $|\phi^3(x)|_5 = 5$, and therefore
$|\phi^n(x)|_5 = 5^{e_n}$ for all $n\geq 3$,
where $e_n = 1 + 3 + \cdots + 3^{n-3} = (3^{n-2}-1)/2$.
Thus,
$\dsps \hat{\lambda}_{\phi,5}(x)
=\lim_{n\to\infty} \frac{e_n}{3^n} \log 5 = \frac{1}{18}\log 5$,
as claimed.
Incidentally, the same argument almost applies to the fifth
point $x=\frac{22}{9}$ as well, except that
$\phi^7(\frac{22}{9})\approx 36.19$.
As a result,
$\hat{\lambda}_{\phi,\infty}(\frac{22}{9})\approx .0014$ is positive;
dividing by $\log(604)$ gives the extra contibution of $.00022$ to
the scaled height.

A similar phenomenon occurs for the map
$\phi(z) = -\frac{1}{240}z^3 + \frac{259}{60}z + 1$, listed
near the bottom of Table~\ref{tab:11pts}.  For that map,
$-12$ and $20$ have the same small height but apparently
disjoint orbits.  The point $20$ maps to $18$, and all three
of $10$, $18$, and $-28$ map to $22$, which then
maps to $12$.  Meanwhile, $-12$ maps to $-22$, which
maps to $-10$; and all three of $-10$, $-18$, and $28$
map to $-20$.
\end{remark}

\begin{table}
\scalebox{0.8}{
\begin{tabular}{|c|c|c|c|}
\hline
$a,b,c$
&
$x, \phi(x), \phi^2(x),\ldots$
&
\parbox{2.4cm}{\hfil other \\ \hfil preperiodic}
&
\parbox{3.4cm}{
\hfil other small height
\\
\makebox[3.4cm][s]{point \hfil  $\frakh$}
}
\\
\hline
$\frac{1}{12},-\frac{25}{12}, 1 \vphantom{\sum_{X_X}^X}$ &
$0$,$\{ 1, -1, 3, -3, 5\}$
& $-5$
& 
\makebox[3.4cm][s]{$-4$: \hfil  $.01595$}
\\
\hline
\hline
$-\frac{3}{2},\frac{7}{6}, 1 \vphantom{\sum_{X_X}^X}$ &
$0$, $1$, $\frac{2}{3}$, $\{\frac{4}{3}, -1\}$
& $\pm\frac{1}{3}$, $-\frac{2}{3}$
& ---
\\
\hline
$\frac{1}{6},-\frac{13}{6}, 1 \vphantom{\sum_{X_X}^X}$ &
$2$, $-2$, $4$, $\{3, -1\}$
& $0$, $1$, $-3$, $-4$
& ---
\\
\hline
$\frac{2}{3},-\frac{13}{6}, 1 \vphantom{\sum_{X_X}^X}$ &
$\frac{3}{2}$, $0$, $1$, $-\frac{1}{2}$, $\{2\}$
& $-2$, $\frac{1}{2}$, $-\frac{3}{2}$
& ---
\\
\hline
$-\frac{4}{3},\frac{13}{12}, 1 \vphantom{\sum_{X_X}^X}$ &
$0$, $1$, $\{\frac{3}{4}, \frac{5}{4}, -\frac{1}{4}\}$
& $-1$, $\frac{1}{4}$, $-\frac{3}{4}$
& ---
\\
\hline
$\frac{3}{4},-\frac{25}{12}, 1 \vphantom{\sum_{X_X}^X}$ &
$\frac{4}{3}$, $0$, $\{ 1, -\frac{1}{3}, \frac{5}{3} \}$
& $\frac{1}{3}$, $-\frac{5}{3}$
& ---
\\
\hline
$-\frac{1}{3},\frac{37}{12}, 1 \vphantom{\sum_{X_X}^X}$ &
$\frac{1}{2}$, $\frac{5}{2}$, $\frac{7}{2}$,
$\{-\frac{5}{2}, -\frac{3}{2}\}$
& $-2$, $-\frac{1}{2}$
& ---
\\
\hline
$-\frac{4}{3},\frac{37}{12}, 1 \vphantom{\sum_{X_X}^X}$ &
$-\frac{5}{4}$, $-\frac{1}{4}$, $\frac{1}{4}$,
$\frac{7}{4}$, $\{ -\frac{3}{4}\}$
& $-1$
& ---
\\
\hline
$\frac{3}{4},-\frac{49}{12}, 1 \vphantom{\sum_{X_X}^X}$ &
$\frac{1}{3}$, $-\frac{1}{3}$, $\frac{7}{3}$, $\{1, -\frac{7}{3} \}$
& $0$, $\frac{4}{3}$, $\frac{5}{3}$, $-\frac{8}{3}$
& ---
\\
\hline
$\frac{6}{5},-\frac{61}{30}, 1 \vphantom{\sum_{X_X}^X}$ &
$-\frac{3}{2}$, $\{0, 1, \frac{1}{6}, \frac{2}{3} \}$
& $\frac{5}{6}$
& ---
\\
\hline
$-\frac{32}{3},\frac{37}{6}, 1 \vphantom{\sum_{X_X}^X}$ &
$-\frac{5}{8}$, $-\frac{1}{4}$, $-\frac{3}{8}$,
$\{-\frac{3}{4}, \frac{7}{8} \}$
& $-\frac{1}{2}$
& ---
\\
\hline
$\frac{1}{48},-\frac{19}{12}, 1 \vphantom{\sum_{X_X}^X}$ &
$2$, $-2$, $4$, $\{-4, 6\}$
& $-6$, $\pm 10$
& ---
\\
\hline
$-\frac{2}{15},\frac{79}{30}, 1 \vphantom{\sum_{X_X}^X}$ &
$-4$, $-1$, $-\frac{3}{2}$, $\{-\frac{5}{2}, -\frac{7}{2}\}$
& $5$
&
\parbox{3.4cm}{
\makebox[3.4cm][s]{$0$: \hfil  $.00688$}
\\
\makebox[3.4cm][s]{$1$:  \hfil  $.02065$}
}
\\
\hline
$-\frac{1}{30},\frac{91}{30}, 1 \vphantom{\sum_{X_X}^X}$ &
$0$, $1$, $\{4, 11, -10\}$
& $-5$, $-6$, $9$
& \makebox[3.4cm][s]{$-4$: \hfil  $.01983$}
\\
\hline
$\frac{8}{15},-\frac{121}{30}, 1 \vphantom{\sum_{X_X}^X}$ &
$\frac{1}{4}$, $0$, $\{1, -\frac{5}{2}, \frac{11}{4}\}$
& $-\frac{11}{4}$
& ---
\\
\hline
$-\frac{49}{48},\frac{31}{12}, 1 \vphantom{\sum_{X_X}^X}$ &
$-\frac{6}{7}$, $-\frac{4}{7}$, $-\frac{2}{7}$,
$\{\frac{2}{7},\frac{12}{7} \}$
& $\pm\frac{10}{7}$, $-\frac{12}{7}$
& ---
\\
\hline
$\frac{5}{48},-\frac{211}{60}, 1 \vphantom{\sum_{X_X}^X}$ &
$\frac{22}{5}$, $-\frac{28}{5}$,
$\{\frac{12}{5},-6,-\frac{2}{5} \}$
& $6$, $\frac{28}{5}$, $\frac{2}{5}$
& ---
\\
\hline
\end{tabular}
}
\caption{Cubic polynomials $az^3 + bz + c$
with a rational preperiodic chain of length $\geq 5$}
\label{tab:5iters}
\end{table}

As mentioned in Remark~\ref{rem:arbitrary},
no preperiodic point in our data took
more than six iterations to produce a repeated value.  In fact,
all but one function required only five.
The one exception is
$\phi(z) = \frac{1}{12}z^3 - \frac{25}{12}z + 1$, for which
the preperiodic point $0$ lands
on the $5$-periodic point $1$ after one iteration.
(There are a total of $8$ small height points for $\phi$,
because $-5$ also maps to $1$, and because
$-4$ has scaled height $.01595\ldots$.)
This map was also the only cubic polynomial in our search
with a rational $5$-periodic point; all other periods
were at most $4$.
Table~\ref{tab:5iters} lists all those cubic polynomials in
our search for which some rational preperiodic point required
$5$ or more iterations to reach a repeat; note that all are of the form
$az^3 + bz + 1$.

\begin{table}
\scalebox{0.8}{
\begin{tabular}{|c|c|c|c|}
\hline
$a,b,c$ & $h(\phi)$
&
$x, \phi(x), \phi^2(x),\ldots$
&
$\frakh(x)$
\\
\hline
$-\frac{25}{24},\frac{97}{24}, 1 \vphantom{\sum_{X_X}^X}$ &
$3.3672$
&
$-\frac{7}{5}$, $-\frac{9}{5}$, $-\frac{1}{5} \vphantom{\sum_{X_X}^X}$,
$\frac{1}{5}$, $\frac{9}{5}$, $\frac{11}{5}$,
$-\frac{6}{5}$, $-\frac{41}{20}$, $\frac{4323}{2560}, \ldots$
&
$.00025$
\\
\hline
$\frac{8}{15},-\frac{289}{120}, 1 \vphantom{\sum_{X_X}^X}$ &
$5.6664$
&
$\frac{5}{8}$, $-\frac{3}{8}$, $\frac{15}{8} \vphantom{\sum_{X_X}^X}$,
$0$, $1$, $-\frac{7}{8}$,
$\frac{11}{4}$, $\frac{175}{32}$, $\frac{307441}{4096}, \ldots$
&
$.00030$
\\
\hline
$-27,\frac{85}{12}, 1 \vphantom{\sum_{X_X}^X}$ &
$5.7807$
&
$-\frac{2}{9}$, $-\frac{5}{18}$, $-\frac{7}{18} \vphantom{\sum_{X_X}^X}$,
$-\frac{1}{6}$, $-\frac{1}{18}$, $\frac{11}{18}$,
$-\frac{5}{6}$, $\frac{193}{18}$, $-\frac{597703}{18}, \ldots$
&
$.00032$
\\
\hline
$-\frac{1}{48},\frac{31}{12}, 1 \vphantom{\sum_{X_X}^X}$ &
$4.8202$
&
$4$, $10$, $6$, $12$, $-4$, $-8$, $-9 \vphantom{\sum_{X_X}^X}$,
$-\frac{113}{16}$, $-\frac{649189}{65536},\ldots$
&
$.00039$
\\
\hline
$-\frac{3}{4},\frac{25}{12}, 1 \vphantom{\sum_{X_X}^X}$ &
$3.2188$
&
$-1$, $-\frac{1}{3}$, $\frac{1}{3}$, $\frac{5}{3} \vphantom{\sum_{X_X}^X}$,
$1$, $\frac{7}{3}$, $-\frac{11}{3}$, $\frac{91}{3}$,
$-\frac{62605}{3}, \ldots$
&
$.00046$
\\
\hline
$\frac{21}{128},-\frac{295}{168}, 1 \vphantom{\sum_{X_X}^X}$ &
$8.4596$
&
$-4$, $-\frac{52}{21}$, $\frac{20}{7}$,
$-\frac{4}{21} \vphantom{\sum_{X_X}^X}$,
$\frac{4}{3}$, $-\frac{20}{21}$, $\frac{124}{49}$,
$-\frac{39572}{50421}, \ldots$
&
$.00047$
\\
\hline
$-\frac{243}{224},\frac{85}{168}, 1 \vphantom{\sum_{X_X}^X}$ &
$6.5917$
&
$-\frac{2}{27}$, $\frac{26}{27}$, $\frac{14}{27}$,
$\frac{10}{9} \vphantom{\sum_{X_X}^X}$,
$\frac{2}{27}$, $\frac{28}{27}$, $\frac{17}{54}$, $\frac{2593}{2304}$,
$\frac{2336653975}{101468602368}, \ldots$
&
$.00057$
\\
\hline
$\frac{4}{21},-\frac{205}{84}, 1 \vphantom{\sum_{X_X}^X}$ &
$5.3230$
&
$\frac{3}{4}$, $-\frac{3}{4}$, $\frac{11}{4}$,
$-\frac{7}{4} \vphantom{\sum_{X_X}^X}$,
$\frac{17}{4}$, $\frac{21}{4}$, $\frac{63}{4}$, $\frac{2827}{4}$,
$\frac{1882717007}{28}, \ldots$
&
$.00058$
\\
\hline
$\frac{15}{8},-\frac{289}{120}, 1 \vphantom{\sum_{X_X}^X}$ &
$5.6664$
&
$\frac{1}{5}$, $\frac{8}{15}$, $0$,
$1  \vphantom{\sum_{X_X}^X}$,
$\frac{7}{15}$, $\frac{1}{15}$, $\frac{21}{25}$, $\frac{276}{3125}$,
$\frac{9626315307}{12207031250}, \ldots$
&
$.00063$
\\
\hline
\end{tabular}
}
\caption{Cubic polynomials $az^3 + bz + c$
with rational points of scaled height
less than $.0007$}
\label{tab:smallht}
\end{table}

Our data supports Conjecture~\ref{conj:ubc} for cubic polynomials
inasmuch as the number of
rational preperiodic points does not grow as $h(\phi)$ increases.  
For example, even though
Table~\ref{tab:11pts} shows a number
of maps $az^3 + bz$ with eleven preperiodic points,
it is important to note that the first
such map had height as small as $h(\phi)=\log(19)$. Similarly,
every preperiodic structure appearing anywhere in
Tables~\ref{tab:11pts}, \ref{tab:9pts0}, and~\ref{tab:9pts1}
appeared already for some map of relatively small height.
That is, the data suggests that all the phenomena that
can occur have already occurred among the small height maps.

In the same way, the data also supports Conjecture~\ref{conj:ht}
for cubic polynomials.  Table~\ref{tab:smallht} lists the only
nine points of scaled height smaller than $.0007$
in our entire search.
(There were only twenty points with scaled height
smaller than $.001$; three of the extra eleven are iterates
of the first three points listed in Table~\ref{tab:smallht}.)
Once again, even though there are two maps of fairly large
height ($\log(289)\approx 5.67$
and $\log(27\cdot 12) \approx 5.78$) with a point of small scaled
height, there was already a map of substantially smaller height
($\log(97)\approx 3.37$) with an even smaller point.

Moreover, the intuition (mentioned in the introduction) that the scaled
height measures the number of iterates required to start the ``explosion''
is on clear display in Table~\ref{tab:smallht}.
For these points,
it takes seven applications of $\phi$ to get to an
iterate with noticeably larger numerator or denominator than its
predecessors.
To get 
a point of smaller scaled height than Table~\ref{tab:smallht}'s
record of $.00025$, then, it seems one
would need a point and map with {\em eight} iterations required to start
the explosion.

Also of note is that, just as in Table~\ref{tab:5iters},
all the maps in Table~\ref{tab:smallht}
are of the form $az^3 + bz + 1$.  In fact, the smallest scaled
height for a map $az^3 + bz$ occurs for
$\pm\frac{5}{3}z^3 \mp \frac{77}{30}z$, at $x=\pm 4/5$.
(Once again, see Remark~\ref{rem:c0form} to explain the four-way tie.)
The scaled height
is $.00591$, more than twenty times as large as the current record
for $az^3+ bz + 1$; indeed, it takes a mere four iterations to land
on $43/40$, at which point the numerator and denominator both
start to explode.

This phenomenon supports the heuristic behind
Conjectures~\ref{conj:ubc} and~\ref{conj:ht}, that it is hard
to have a lot of points of small height, as follows.
If $x$ were a small height point for $az^3 + bz$,
then $-x$ would have the same small height; their iterates
would also have (not quite as) small heights, too.
Together with the fixed point at $0$, then, there would be
more small height points than the heuristic would say are allowed.
This idea is further supported by
Tables~\ref{tab:11pts}, \ref{tab:9pts0}, and~\ref{tab:9pts1}: 
while
it is possible to have eleven preperiodic points or ten
points of small height, or even some of each,
it does not seem possible to have more than eleven 
total such points.
Thus, there seems to be an upper bound for the total number of points
of small height, as predicted by
Conjectures~\ref{conj:ubc} and~\ref{conj:ht}.


{\bf Acknowledgements.}
The results of this paper came from an REU in
Summer 2007 at Amherst College, supported by the NSF
and Amherst College.  
Authors Dickman, Joseph, and Rubin, who were the REU
participants, as well as Benedetto, who supervised the
project, gratefully acknowledge the support of
NSF grant DMS-0600878.  The remaining authors
gratefully acknowledge the support of Amherst College:
the Schupf Scholar Program for Krause, and Dean of Faculty
student funds for Zhou.
The authors also thank Amherst College, and especially
Scott Kaplan, for the use of the College's high-performance
computing cluster, on which we ran our computations.
Finally, we thank Patrick Ingram,
Bjorn Poonen, Joseph Silverman, Franco Vivaldi,
and the referee
for their helpful comments
on the exposition.

\end{document}